\documentclass[12pt]{amsart}

\usepackage[matrix,arrow,curve,frame]{xy}
\usepackage{amsmath,amsthm,amssymb,enumerate,stmaryrd}
\usepackage{latexsym}
\usepackage{amscd}
\usepackage[colorlinks=true]{hyperref}
\usepackage{euscript}
\usepackage{url}
\usepackage{mathrsfs}
\newtheorem*{Thm}{Theorem}
\newtheorem{thm}[subsection]{Theorem}
\newtheorem{defn}[subsection]{Definition}
\newtheorem*{Defi}{Definition}
\newtheorem*{Remark}{Remark}

\newtheorem*{Conj}{Conjecture}

\newtheorem{claim}[subsection]{Claim}
\newtheorem{corr}[subsection]{Corollary}

\newtheorem{conj}[subsection]{Conjecture}
\newtheorem{remark}[subsection]{Remark}

\theoremstyle{definition}
\newtheorem{example}[subsection]{Example}

\newcommand{\cat}{\mathcal}
\newcommand{\lra}{\longrightarrow}

\newcommand{\llra}[1]{\stackrel{#1}{\lra}}

\newcommand{\R}{\mathbb R}
\newcommand{\Z}{\mathbb Z}
\newcommand{\F}{\mathbb F}

\newcommand{\C}{\mathbb C}

\newcommand{\T}{\mathbb T}
\DeclareMathOperator{\Q}{\mathbb{Q}}

\DeclareMathOperator{\Ks}{\mathscr{K}}

\newcommand{\gA}{{\cat A}}
\newcommand{\gSA}{{SA}}
\newcommand{\gF}{{\cat F}}
\newcommand{\gH}{{\cat H}}

\newcommand{\I}{{\cat I}}

\newcommand{\sBC}{s\mathscr{B}}
\newcommand{\sBCU}{s\mathscr{B}_\infty}
\newcommand{\sNBC}{s\mathscr{B}}
\newcommand{\BC}{\mathscr{B}}
\newcommand{\UC}{\mathscr{U}}
\newcommand{\BSa}{\cat Bt\cat S}

\newcommand{\sBS}{s\cat B\cat Sh}
\newcommand{\BS}{\cat B\cat Sh}
\newcommand{\Sh}{\cat Sh}

\DeclareMathOperator{\OSU}{\Omega U}

\DeclareMathOperator{\OSSU}{\Omega SU}

\DeclareMathOperator{\LG}{LG}

\DeclareMathOperator{\Map}{Map}

\DeclareMathOperator{\Br}{Br}
\DeclareMathOperator{\Gr}{Gr}

\DeclareMathOperator{\Tr}{Trig}
\DeclareMathOperator{\E}{E}
\DeclareMathOperator{\Ev}{\cat E}

\DeclareMathOperator{\LgG}{LG}
\DeclareMathOperator{\Ho}{\mbox{H}}

\DeclareMathOperator{\SU}{U}
\DeclareMathOperator{\ESU}{EU}
\DeclareMathOperator{\ESUrT}{EU(r)/T}
\DeclareMathOperator{\ET}{ET}

\DeclareMathOperator{\ESUrGi}{EU(r)/G_i}
\DeclareMathOperator{\ZSU}{ZU}

\DeclareMathOperator{\BSU}{BU}
\DeclareMathOperator{\MU}{MU}
\DeclareMathOperator{\BT}{BT}
\DeclareMathOperator{\LSU}{LU}
\DeclareMathOperator{\LBSU}{LBU}

\DeclareMathOperator{\SSU}{SU}
\DeclareMathOperator{\LSSU}{LSU}

\DeclareMathOperator{\PU}{PU}
\DeclareMathOperator{\BGL}{BGL}

\DeclareMathOperator{\Aut}{Aut}
\DeclareMathOperator{\GL}{GL}

\DeclareMathOperator{\Hh}{\mathscr{H}}
\DeclareMathOperator{\Ah}{\mathscr{A}}
\DeclareMathOperator{\LV}{\mathcal{L}}

\newfont{\german}{eufm10}

  \DeclareMathOperator{\hocolim}{hocolim}

\newcommand\qu{/\kern-.7ex/}
 \newcommand\bl{[ \kern-.2ex[}
  \newcommand\br{] \kern-.2ex]}

 \newcommand{\iso} {\cong}

\begin{document}
\pagestyle{plain}

\title
{Symmetry Breaking and Link homologies II}
\author{Nitu Kitchloo}
\address{Department of Mathematics, Johns Hopkins University, Baltimore, USA}
\email{nitu@math.jhu.edu}
\maketitle

\begin{abstract}
In the first part of this paper, we constructed a filtered $\SU(r)$-equivariant stable homotopy type called the spectrum of strict broken symmetries $\sBC(L)$ of links $L$ given by closing a braid with $r$ strands. Evaluating this filtered spectrum on suitable $\SU(r)$-equivariant cohomology theories gives rise to a spectral sequence of link invariants that converges to the cohomology of the limiting spectrum $\sBCU(L)$. In this paper, we evaluate our construction on Borel equivariant singular cohomology $\Ho_{\SU(r)}^\ast$. We show that the $E_2$-term of the spectral sequence is isomorphic to an unreduced verision of triply-graded link homology. More precisely, we show that the $E_1$-term of the spectral sequence is isomorphic to the Hochschild-homology of Soergel bimodules that was shown by M. Khovanov in \cite{K} to compute triply-graded link homology. We also set up the theory that allows for evaluating our construction on equivariant cohomologies that are twisted by adjoint-equivariant local systems on $\SU(r)$. This allows us to apply Borel equivariant cohomology twisted by a universal power series of the form $\wp(x) = \sum_{i \geq 1} b_i x^i$ with $b_i$ denoting formal variables. The specialization to $\wp(x) = x^n$ gives rise to a link homology that has recently been shown by T. Mej\'{i}a Gomez \cite{Go} to be isomorphic to $sl(n)$-link homology. In other words, $\wp(x) = \sum_{i \geq 1} b_i x^i$ can be interpreted as the (differential of the) universal potential function with no linear term, in the language of matrix factorizations \cite{KR2}.
\end{abstract}

\tableofcontents

\section{Introduction}

\medskip
\noindent
In \cite{Ki1}, we constructed a $\SU(r)$-equivariant filtered homotopy type $\sBC(L)$, called the spectrum of strict broken symmetries, which was an invariant of links $L$. The aim of this article is to apply Borel equivariant singular cohomology $\Ho_{\SU(r)}^\ast$ to the above construction, and to invoke the filtration to set up a spectral sequence that converges to the cohomology of the limiting value $\sBCU(L)$ of the fitration. We also allow for twistings on Borel equivariant cohomology. The pages of the spectral sequence $E_q(L)$ for $q \geq 2$ recover important link invariants. In the untwisted case, we identify the $E_2$-term of our spectral sequence with an unreduced, integral form of triply-graded link homology \cite{K}. Indeed, we show that the $E_1$-term of our spectral sequence can be identified with the complex of Hochschild homologies of Bott-Samelson Soergel bimodules associated to braid words, so that the cohomology of the resulting complex computes triply graded link homology as shown by M. Khovanov in \cite{K}. In the twisted case, we conjecture that one recovers $sl(n)$-link homology for any $n$.

\medskip
\noindent
Let us now recall the definition of the spectrum of strict broken symmetries $\sBC(L)$ as defined in \cite{Ki1}, so as to apply the above construction. In order to make this definition precise, consider a braid element $w \in \Br(r)$, whose closure is the link $L$, and where $\Br(r)$ stands for the braid group on $r$-strands. For the sake of exposition, in this introduction we only consider the case of a positive braid that can be expressed in terms of positive exponents of the elementary braids $\sigma_i$ for $i < r$. The general case will be described later. 

\smallskip
\noindent
Let $I = \{ i_1, i_2, \ldots, i_k \}$ denote an indexing sequence with $i_j < r$, so that a positive braid $w$ admits a presentation in terms of the fundamental generators of the $r$-stranded braid group $\Br(r)$, $w = w_I := \sigma_{i_1} \sigma_{i_1} \ldots \sigma_{i_k}$. Let $T$ denote the standard maximal torus, and let $G_i$ denote the unitary form in the reductive Levi subgroup having roots $\pm \alpha_i$. We consider $G_i$ as a two-sided $T$-space under the left(resp. right) multiplication. 

\medskip
\noindent
The equivariant $\SU(r)$-spectrum of broken symmetries is defined as the (suspension) spectrum corresponding to the $\SU(r)$-space $\BC(w_I)$
\[ \BC(w_I) := \SU(r) \times_T ({G_{i_1}} \times_T {G_{i_2}} \times_T \cdots \times_T {G_{i_k}}) = \SU(r) \times_T \BC_T(w_I) , \]
with the $T$-action on $\BC_T(w_I) := (G_{i_1} \times_T G_{i_2} \times_T \cdots \times_T G_{i_k})$ given by endpoint conjugation
\[ t \,  [(g_1, g_2, \cdots, g_{k-1}, g_k)] := [(tg_1, g_2, \cdots, g_{k-1}, g_kt^{-1})]. \]

\noindent
The $\SU(r)$-stack $\SU(r) \times_T (G_{i_1} \times_T G_{i_2} \times_T \cdots \times_T G_{i_k})$ is equivalent to the stack of principal $\SU(r)$-connections on the trivial $\SU(r)$-bundle over $S^1$, endowed with a reduction of the structure group to $T$ at $k$ distinct points, and so that the holonomy between successive points belongs the corresponding group of the form $G_i$ in terms of this reduction. 

\bigskip
\begin{Defi} (Strict broken symmetries and their normalization), (\cite{Ki1} definition 2.8)

\noindent
Let $L$ denote a link described by the closure of a positive braid $w \in \Br(r)$ with $r$-strands, and let $w_I$ be a presentation of $w$ as $w = \sigma_{i_1} \ldots \sigma_{i_k}$. We first define the limiting $\SU(r)$-spectrum $\sBCU(w_{I})$ of strict broken symmetries as the space that fits into a cofiber sequence of $\SU(r)$-spaces: 
\[ \hocolim_{J \in \I} \BC(w_{J}) \longrightarrow \BC(w_I) \longrightarrow \sBCU(w_{I}). \]
where $\I$ is the category of all proper subsets of $I = \{ i_1, i_2, \ldots, i_k \}$. 

\medskip
\noindent
The spectrum $\sBCU(w_I)$ admits a natural increasing filtration by spaces $F_t \, \sBC(w_I)$ defined as the cofiber on restricting the above homotopy colimit to the full subcategories $\I^t \subseteq \I$ generated by subsets of cardinality at least $(k-t)$, so that the lowest filtration is given by $F_0 \sBC(w_I) = \BC(w_I)$. 

\medskip
\noindent
Define the spectrum of strict broken symmetries $\sBC(w_I)$ to be the filtered spectrum $F_t \, \sBC(w_I)$ above. The normalized spectrum of strict broken symmetries of the link $L$ is defined as 
\[ \sNBC(L) := \Sigma^{-2k} \sBC(w_I). \]
\end{Defi}

\medskip
\noindent
In (\cite{Ki1} theorem 8.5) we  proved the general form of the following result. 

\medskip
\begin{Thm}
As a function of (framed) links $L$\footnote{see \cite{Ki1}, (Section 8) for discussion on the framing enhancement} , the filtered $\SU(r)$-spectrum of strict broken symmetries $\sBC(L)$ is well-defined up to quasi-equivalence (\cite{Ki1} definition 3.4). In particular, the limiting equivariant stable homotopy type $\sBCU(L)$ is a well-defined (framed) link invariant in $\SU(r)$-equivariant spectra (see \cite{Ki1} remark 3.5, and theorem below). 
\end{Thm}

\medskip
\noindent
In Section \ref{DomK2} of this paper we will prove, and then generalize the following theorem to arbitrary braids. Similar ideas have been explored in \cite{WW}.

\medskip
\begin{Thm}
Given a link $L$ given by the closure of a positive braid word $w_I$ of length $k$ in $r$-strands as above, then evaluating equivariant singular cohomology on the filtered spectrum $\sBC(L)$ gives rise to a cohomologically graded spectral sequence 
\[ E_1^{t,s} = \bigoplus_{J \in \I^t/\I^{t-1}}    \Ho_{\SU(r)}^s(\BC(w_J))  \, \, \Rightarrow \, \, \Ho_{\SU(r)}^{s+t-2k}(\sBCU(L)). \]
The terms $ \Ho_{\SU(r)}^\ast(\BC(w_J))$ are isomorphic to the Hoschschild homology of Soergel bimodules indexed by subsets $J \subseteq I$, and the differential $d_1$ is induced by inclusion of subsets. It follows from \cite{{K, Ka}} that the $E_2$-term is isomorphic to an integral form of triply-graded link homology, whose value on the unknot has the form $\Z[x] \otimes \Lambda (y)$, with the degrees of $x$ and $y$ being $2$ and $1$ resp. Furthermore, the terms $E_q(L)$ are invariants of the link $L$ for all $q \geq 2$. In particular, $\Ho_{\SU(r)}^\ast(\sBCU(L))$ is an integral version of the invariant Rasmussen calls $E_\infty(-1)$ (see Theorem 3 in \cite{J2}), that only records the number of components of $L$.
\end{Thm}

\medskip
\begin{Remark}
In (\cite{Ki1} see remark 6.6), we showed that the filtered spectrum $\sBC(w_I) \wedge_T S^0$ obtained by taking the orbits under the right $T$-action is invariant under the Braid and inverse relations and represents a filtered equivariant homotopy type for the complex of Bott-Samelson Soergel bimodules studied by Rouquier in \cite{R,R2}. 
\end{Remark}

\smallskip
\noindent
In order to describe twistings of equivariant cohomology theories so as to make a connection with $sl(n)$-link homology, let us start by noticing that the spaces of broken symmetries $\BC(w_I)$ admit a canonical equivariant map to conjugation action of $\SU(r)$ on itself
\[ \rho_I : \BC(w_I) \longrightarrow \SU(r), \quad \quad [(g, g_1, \ldots, g_k)] \longmapsto g(g_1 g_2 \ldots g_k) g^{-1}. \]
Thinking of the co-domain $\SU(r)$ as the space of principal $\SU(r)$-connections on the circle, we see that the map $\rho_I$ is the map that simply composes the holonomies of all the broken symmetries along the circle. It follows that $\rho_I$ is compatible under inclusions of subsets $J \subseteq I$. We show in section \ref{LiftsSBS} that this compatibility allows us to use the derived equivariant local system with values in $\SU(r)$ to define twists $\theta$ of any suitable family $\E_{\SU(r)}$ of equivariant cohomology theories. We also study the action of symmetries and cohomology operations on these twisted theories. 

\medskip
\noindent
In sections \ref{LiftsSBS} and \ref{HT}, we also describe a twist of Borel equivariant cohomology, by first extending it to a multiplicative equivariant theory $\Hh$, so that $\Hh_{\SU(r)}$ is obtained from Borel equivariant singular cohomology $\Ho_{\SU(r)}$ by adjoining formal variables $b_i$ in cohomological degree $-2i$. Let $\wp(x) = \sum_{i \geq 1} b_i x^i$ be the formal power series with the coefficients being the variables $b_i$ with $x$ having cohomological degree $2$. We may use the above framework to describe a twist $\theta_\wp$ of $\Hh$ denoted by $\sideset{^{\theta}}{^\ast}\Hh$. In section \ref{RSS} we use a filtration of $\sideset{^{\theta}}{^\ast}\Hh$ by powers of the ideal of indecomposables in $b_1, b_2, \ldots$, so that the total filtration of $\sideset{^{\theta}}{^\ast}\Hh_{\SU(r)}(\sBC(L))$ gives rise to a family of link homology theories defined as the $E_q$-terms of the corresponding spectral sequence for any $q \geq 2$. Furthermore the first differential $d_1$ is a sum of two differentials, the first being induced by the inclusion of subsets $J$ as before, and the second induced by the filtration of $\sideset{^{\theta}}{^\ast}\Hh$ by powers of the ideal of indecomposables in $b_1, b_2, \ldots$. For any $q \geq 2$, the value of these theories on unknot $L$ is 
\[  E^{0,\ast}_q(L, \theta_{\wp}) = E^{0,\ast}_2(L, \theta_{\wp}) = \frac{\Z[b_1, b_2, \ldots] \llbracket x \rrbracket}{\langle \wp(x) \rangle}, \]
where $\Z[b_1, b_2, \ldots] \llbracket x \rrbracket$ denotes graded power series in $x$ with coefficients in the variables $b_i$. This appears to suggest that the $E_2(L, \theta_{\wp})$-term gives rise to the $sl(n)$-link homology with respect to the universal matrix factorization whose potential has the differential $\wp(x)$. 

\medskip
\begin{Conj}
The algebraic specialization $\wp(x) = x^n$ of the $E_1(L, \theta_\wp)$ term above leads to homology groups that are isomorphic to $sl(n)$-link homology. See \ref{Mainconj} for more details.
\end{Conj}

\noindent
Expressing $E_1(L,\theta_{\wp})$ as a bicomplex gives rise to a topological version of the Rasmussen spectral sequence (see Theorem 2 in \cite{J2}), that begins with the Triply graded link homology of $L$ and converges to $E_2(L, \theta_{\wp})$ above. In \cite{Go}, T. Mej\'{i}a Gomez has compared the topological and algebraic Rasmussen spectral sequences to resolve the above conjecture in the affirmative. 

\medskip
\noindent
\section{Borel equivariant cohomology of Broken symmetries and Soergel bimodules} \label{DomK2} 

\noindent
Before we start with the main results as stated in the previous section, let us first recall the general formulation of the statements of the definitions and theorems for arbitrary links $L$. Let $T = T^r \subseteq \SU(r)$ denote the standard maximal torus, and let $\Br(r)$ denote the braid group generated by the standard braids $\sigma_i, 1 \leq i < r$. The weights of $T$ will be denoted by $\sum_{i \leq r} \Z\langle x_i \rangle$, and will be identified with $\Ho_T^2$, so that the simple roots $\alpha_i$ are expressed as $x_i - x_{i+1}$. Let $h_i, 1 \leq i < r$ denote the co-roots, so that $\alpha_j(h_i) = a_{ij}$ are the entries in the Cartan matrix for $\SU(r)$. In this article, we often require a basis of weights constructed out of {\em dual} co-roots. This basis $\{ h_i^\ast, 1 \leq i \leq r \}$ is defined in term of the generators $x_j$ as 
\[ h_i^\ast = \sum_{j \leq i} x_j, \quad \quad \mbox{in particular, we have} \quad \quad h_i^\ast(h_j) = \delta_{i,j} \quad \mbox{for} \quad j < r. \]
Identifying $\Ho_T^2$ with $\Ho^1(T)$ via the transgression map, the collection of elements $h_i^\ast$ as well as the collection $x_i $ yield a $\Z$-basis of $\Ho^1(T,\Z)$. We will use the notation $\hat{x}_i$ to indicate the corresponding basis of $\Ho^1(T)$ if it is not clear from context. The action of the Weyl group on $\Ho_T^2$, with generators also denoted by $\sigma_i$ (since the group is clear from context) is
\[ \sigma_i \alpha := \alpha - \alpha(h_i) \alpha_i. \]
Notice that by definition $h_r^\ast$ is the central character that is invariant under the Weyl group. Let $G_i \subseteq \SU(r)$ denote the unitary form in the reductive Levi subgroup generated by the roots $\pm \alpha_i$. Let $\zeta_i$ denote the virtual $G_i$ representation $(\mathfrak{g}_i-r\R)$, where $\mathfrak{g}_i$ is the adjoint representation of $G_i$ denoted by $Ad$, and $r\R$ is the trivial representation of dimension $r$. Notice that the restriction of $\zeta_i$ to $T$ is isomorphic to the root space representation $\alpha_i$ (as a real representation). Let $S^{-\zeta_i}$ denote the sphere spectrum for the virtual $G_i$ representation $-\zeta_i$ defined as in (\cite{Ki1} definition 2.3). Consider a general indexing sequence for arbitrary braid words $I := \{ \epsilon_{i_1} i_1, \cdots, \epsilon_{i_k} i_k \}$, where $i_j < r$, and $\epsilon_j = \pm 1$. Assume that $w = w_I := \sigma_{i_1}^{\epsilon_{i_1}} \cdots \sigma_{i_k}^{\epsilon_{i_k}}$. Recall the spectrum $\BC(w_{I}) := \SU(r)_+ \wedge_T \BC_T(w_I)$ (\cite{Ki1} definition 2.4),  with
\[  \BC_T(w_{I}) := H_{i_1} \wedge_T \ldots  \wedge_T H_{i_k},  \, \, \mbox{and} \, \, H_i = S^{-\zeta_i} \wedge \, {G_i}_+,  \, \, \mbox{if} \, \, \epsilon_i = -1, \, \,  H_i = {G_i}_+ \, \,  \mbox{else}. \]
The $T \times T$-action on $H_i$ is defined by demanding that an element $(t_1, t_2) \in T \times T$ acts on $S^{-\zeta_i} \wedge \, {G_i}_+$ by smashing the action $Ad(t_1)_\ast$ on $S^{-\zeta_i}$ with the standard $T \times T$ action on ${G_i}_+$ given by left (resp. right) multiplication. As before the $T$-action on $\BC_T(w_I)$ is by conjugation on the first and last factor. It is clear that each bundle $\zeta_i$ above represents a $\SU(r)$-equivariant vector bundle over $\SU(r) \times_T(G_{i_1} \times_T \cdots \times_T G_{i_k}) := \BC(w_{\underline{I}})$, where $\underline{I}$ is the indexing set obtained from $I$ by replacing each $\epsilon_j$ with $1$. By construction $\BC(w_{I})$ is the corresponding equivariant Thom spectrum over $\BC(w_{\underline{I}})$
\[ \BC(w_I) = \BC(w_{\underline{I}})^{-\zeta_I}, \quad \mbox{where} \quad \zeta_I := \bigoplus_{i_j \in I | \epsilon_{i_j} = -1} \zeta_{i_j}. \]

\begin{defn} (The functor $\BC(w_I)$, \cite{Ki1} defintion 2.6) \label{DomK3-4}

\noindent
Given a braid word $w_{I}$, for $I = \{ \epsilon_{i_1} i_1, \cdots, \epsilon_{i_k} i_k \}$, let $2^{I}$ denote the set of all subsets of $I$. Let us define a poset structure on $2^{I}$ generated by demanding that nontrivial indecomposable morphisms $J \rightarrow K$ have the form where either $J$ is obtained from $K$ by dropping an entry $i_j \in K$ (i.e. an entry for which $\epsilon_{i_j} = 1$), or that $K$ is obtained from $J$ by dropping an entry $-i_j$ (i.e an entry for which $\epsilon_{i_j} = -1$). The construction $\BC(w_{J})$ induces a functor from the category $2^{I}$ to $\SU(r)$-spectra. More precisely, given a nontrivial indecomposable morphism $J \rightarrow K$ obtained by dropping $-i_j$ from $J$, the induced map $\BC(w_{J}) \rightarrow \BC(w_{K})$ is obtained by applying the map of (\cite{Ki1} claim 2.5). Likewise, if $J$ is obtained from $K$ by dropping the factor $i_j$, then the map $\BC(w_{J}) \rightarrow \BC(w_{K})$ is defined as the canonical inclusion. 
\end{defn}

\begin{defn} \label{DomK3-5} (Strict broken symmetries, \cite{Ki1} definitions 2.7, 2.8, 2.10) 

\noindent
let $I^+ \subseteq I$ denote the terminal object of $2^{I}$ given by dropping all terms $-i_j$ from $I$ (i.e. terms for which $\epsilon_{i_j} = -1$). Define the poset category $\I$ to the subcategory of $2^{I}$ given by removing $I^+$
\[ \I = \{ J \in 2^{I}, \,  J \neq I^+ \}. \]
\noindent
The filtered $\SU(r)$-spectrum $F_t \sBC(w_{I})$ of strict broken symmetries is defined via the cofiber sequence of equivariant $\SU(r)$-spectra
\[ \hocolim_{J \in \I^t} \BC(w_{J}) \longrightarrow \BC(w_{I^+}) \longrightarrow F_t \, \sBC(w_{I}), \]
where $\I^t \subseteq \I$ consisting of objects no more than $t$ nontrivial composable morphisms away from $I^+$. We define $F_0 \, \sBC(w_{I}) = \BC(w_{I^+})$, and $F_k = \ast$ for $k < 0$. We normalize the spectrum as 
\[ \sNBC(L) := \Sigma^{l(w_I)} \sBC(w_I)[\varrho_I], \]
where $\Sigma^{l(w_I)}$ denotes the suspension by $l(w_I) := l_-(w_I) - 2l_+(w_I)$ with $l_+(w_I)$ being the number of positive and $l_-(w_I)$ being the number of negative exponents in the presentation $w_I$ for $w$ in terms of the generators $\sigma_i$. Also, $\sBC(w_I)[\varrho_I]$ denotes the filtered spectrum $\sBC(w_I)$ with a shift in indexing given by $F_t \, \sBC(w_I)[\varrho_I] := F_{t+\varrho_I} \, \sBC(w_I)$, with $\varrho_I$ being one-half the difference between the cardinality of the set $I$, denoted by $|I|$, and the mimimal word length $|w|$ of $w \in \Br(r)$,
\[ \varrho_I = \frac{1}{2} (|I| - |w|). \]
\end{defn}

\noindent
Our goal now is to apply Borel equivariant cohomology $\Ho_{\SU(r)}^\ast$ to our construction. 

\begin{defn} (Borel equivariant singular cohomology) \label{DomK3-5a}

\noindent
Borel equivariant cohomology of a $\SU(r)$-spectrum $X$ is defined as the singular cohomology of the Borel construction $X \wedge_{\SU(r)} \ESU(r)_+$, where $\ESU(r)$ denotes the free contractible $\SU(r)$-CW complex. 
\end{defn}


\medskip
\noindent
\begin{defn} (Bott-Samelson varieties) \label{DomK2-10b}

\noindent
Let $I = \{i_1, \ldots, i_k\}$. Define the Bott-Samelson variety $\BSa(w_I)$ as
\[ \BSa(w_I) = \BC(w_I)/T = \SU(r) \times_T (G_{i_1} \times_T \cdots \times_T G_{i_k}/T). \]
Notice that $\BSa(w_I)$ supports two equivariant maps to $\SU(r)/T$ defined as
\[ \pi(I) : \BSa(w_I) \longrightarrow \SU(r)/T, \quad [(g, g_1, g_2, \cdots, g_{k-1}, g_k)] \longmapsto gT, \]
\[ \quad \quad \quad \, \, \tau(I) : \BSa(w_I) \longrightarrow \SU(r)/T, \quad [(g, g_1, g_2, \cdots, g_{k-1}, g_k)] \longmapsto gg_1\ldots g_k T.\]
These maps endow $\Ho_{\SU(r)}^\ast(\BSa(w_I))$ with a bimodule structure over the coefficients $\Ho_T^\ast$. 
\end{defn}

\medskip
\noindent
Let us now explore how this relates to Soergel bimodules. 

\medskip
\begin{thm} \label{DomK2-11a}
Given a positive indexing sequence $I = \{ i_1, \ldots, i_k \}$, the equivariant cohomology $\Ho_{\SU(r)}^\ast(\BSa(w_I))$ is isomorphic as bimodules to the ``Bott-Samelson" Soergel bimodule \cite{Le} given by a tensor product
\[ \Ho_{\SU(r)}^\ast(\BSa(\sigma_{i_1})) \otimes_{\Ho_T^\ast} \ldots \otimes_{\Ho_T^\ast} \Ho_{\SU(r)}^\ast(\BSa(\sigma_{i_k})). \]
Furthermore, each term $\Ho_{\SU(r)}^\ast(\BSa(\sigma_i))$ is isomorphic to the bimodule
\[ \Ho_{\SU(r)}^\ast(\BSa(\sigma_i)) = \Ho_T^\ast \otimes_{\Ho_{G_i}^\ast} \Ho_T^\ast, \]
where $\Ho_{G_i}^\ast$ denotes the $G_i$-equivariant cohomology of a point. Finally, the equivariant cohomology $\Ho_{\SU(r)}^\ast(\BC(w_I))$ is isomorphic as $\Ho_T^\ast$-modules to the Hochschild homology of the bimodule $\Ho_{\SU(r)}^\ast(\BSa(w_I))$ with coefficients in the bimodule $\Ho_T^\ast$. Furthermore, this isomorphism is canonical if one works over $\Q$. 
\end{thm}

\begin{proof}
Let us begin by noticing that one has a pullback diagram
\[
\xymatrix{
\ESU(r) \times_{T} (G_i/T)   \ar[d]^{\pi} \ar[r]^{\quad \tau} &  \ESU(r)/T \ar[d]^{\pi} \\
\ESU(r)/T      \ar[r]^{\tau} & \ESU(r)/G_i.
}
\]
where $\tau$ is induced by the right action of $G_i$ on $\ESU(r)$, and the map $\pi$ by the projection $G_i/T \longrightarrow pt$. Now all the spaces in the above diagram have torsion free, evenly graded cohomology, and so we can invoke the Eilenberg-Moore spectral sequence \cite{S} to calculate the cohomology of the pullback. The spectral sequence collapses to the above tensor product of the two copies of $\Ho_T^\ast$ over $\Ho_{G_i}^\ast$ since $\Ho_T^\ast$ is a free module over $\Ho_{G_i}^\ast$. 

\medskip
\noindent
Notice also that $\Ho_{\SU(r)}^\ast(\BSa(\sigma_i))$ is a free $\Ho_T^\ast$ module under either the left or the right module structure as can be verified directly, or by using the fact that the Serre spectral sequence of the fibration given by the map $\pi$ collapses (since both the fiber and base have evenly graded cohomology).

\medskip
\noindent
Now consider the pullback diagram of definition \ref{Bott-Samelson}. Inductively using freeness over $\Ho_T^\ast$, and the Eilenberg-Moore spectral sequence, we get the iterated tensor product decomposition for $\Ho_{\SU(r)}^\ast(\BSa(w_I))$ as claimed in the statement of the theorem. 
It remains to identify $\Ho_{\SU(r)}^\ast(\BC(w_I))$ with the Hochschild homology of $\Ho_{\SU(r)}^\ast(\BSa(w_I))$ with coefficients in $\Ho_T^\ast$. For this, consider the principal $T$-fibration given by the canonical projection
\[ \ESU(r) \times_{\SU(r)} \BC(w_I) \longrightarrow \ESU(r) \times_{\SU(r)} \BSa(w_I). \]
It is easy to see that this fibratrion is classified by the map
\[ \ESU(r) \times_{\SU(r)} \BSa(w_I) \llra{\pi(I) \times \tau(I)} \BT \times \BT \longrightarrow \BT, \]
where the first map is given by the product of the maps defined in definition \ref{DomK2-10b} that endow $\Ho_{\SU(r)}^\ast(\BC(w_I))$ with a bimodule structure, and the second map represents the map that classifies the difference map $T \times T \rightarrow T$, $(s,t) \longmapsto st^{-1}$. 

\medskip
\noindent
We now consider the Serre spectral sequence for the fibration defined above. Let us pick generators $\{ \epsilon_1, \ldots, \epsilon_r \} \subset \Ho^1(T, \Z)$ so that $\epsilon_i$ is identified with the dual co-root $h_i^\ast$ (defined at the beginning of this section) under the identification $\Ho^1(T) = \Ho_T^2$. The $E_2$-term of the Serre spectral sequence computing $\Ho_{\SU(r)}^\ast(\BC(w_I))$ is given by
\[ \Lambda(\epsilon_1, \ldots, \epsilon_r) \otimes \Ho_{\SU(r)}^\ast(\BSa(w_I)), \quad d(\epsilon_k) = \pi(I)^\ast(h_k^\ast) - \tau(I)^\ast(h^\ast_k). \]
This complex is the standard Koszul complex that computes the Hochschild homology of the bimodule $\Ho_{\SU(r)}^\ast(\BSa(w_I))$ with coefficients in $\Ho_T^\ast$. This Hochschild homology can be shown to be a free module over $\Ho_T^\ast$ by an induction argument using braid invariance and the Markov moves (see Lemma 3.5 and Corollary 3.6 in \cite{GHMN}, which can easily be checked to also work over $\Z$). Next, we show that the Serre spectral sequence above collapses at $E_3$\footnote{this is a special case of a conjecture of Kumar \cite{Ku}(11.5.18)}. Furthermore, using freeness over $\Ho_T^\ast$, all extensions problems in the spectral sequence can be solved. 

\medskip
\noindent
In order to establish the collapse at $E_3$, let us begin by recalling that the preceding discussion describes the space $\ESU(r) \times_{\SU(r)} \BC(w_I)$ inductively via homotopy pullbacks along maps between classifying spaces of compact Lie groups of semisimple rank one. In particular, given any odd integer $k$, one has an unstable Adams operation $\psi_k$ that acts on $\Ho^\ast_{\SU(r)} (\BC(w_I))$ \cite{Wi}. By the naturality of $\psi_k$, we see that it induces an action on the Serre spectral sequence above. Since the fiber and base of the Serre spectral sequence are generated by degree one (resp. two) classes, one can show that in the $(r,s)$-grading of the spectral sequence, $\psi_k$ acts by scalar multiplication with $k^{s+r/2}$. Since $d_n$ commutes with $\psi_k$, a trivial diagram chase now forces $d_n$ to be trivial for $n>2$ for grading reasons. 

\smallskip
\noindent
Finally, assume one is working with coefficients in $\Q$. We may decompose the groups $\Ho_{\SU(r)}^\ast(\BC(w_I), \Q)$ into the distinct eigenspaces of $\psi_k$. The above discussion shows that these eigenspaces  may be canonically identified with elements in homogeneous Hochschild bidegree as described above. In other words, the extension problems of the Serre spectral sequence described above can be {\em canonically} solved over $\Q$. 
\end{proof}

\medskip
\begin{remark} \label{Galois3}
The Galois symmetry $\sigma$ (\cite{Ki1} Section 8) acting on $\Ho_{\SU(r)}^\ast(\BC(w_I))$ extends the Adams operation $\psi_k$ used above. The action of $\sigma$ can be described in terms of the Koszul complex, where $\sigma$ acts as $-1$ on $\Ho^1(T)$ and $\Ho_{\SU(r)}^2(\BSa(w_I))$. 
\end{remark}

\medskip
\begin{remark} \label{DomK2-11b}
In (\cite{Ki1} see remark 6.6), we pointed out that the filtered equivariant homotopy type $\sBC(w_I) \wedge_T S^0$ is independent of the braid and inversion relations, where we take the $T$-orbits under the right $T$-action on $\sBC(w_I)$. In particular, it lifts the chain complexes of Bott-Samelson Soergel bimodules studied by Rouquier in \cite{R,R2}. 
\end{remark}

\medskip
\noindent
It remains to establish the existence of the spectral sequence, whose $E_1$-term is isomorphic to an integral form of Hochschild homologies of Soergel bimodules. 

\medskip
\begin{thm} \label{DomK2-11c}
One has a cohomologically graded spectral sequence with $E_1$-term given by 
\[ E_1^{t,s} = \bigoplus_{J \in \I^t/\I^{t-1}}    \Ho_{\SU(r)}^s(\BC(w_J))  \, \, \Rightarrow \, \, \Ho_{\SU(r)}^{s+t+l(w_I)}(\sBCU(L)). \]
The differential $d_1$ is the canonical simplicial differential induced by the functor in (\cite{Ki1} definition 2.6). In addition, the terms $E_q(L)$ are invariants of the link $L$ for all $q \geq 2$. Furthermore, the $E_2$-term of this spectral sequence is isomorphic to an unreduced, integral form of triply-graded link homology as defined in \cite{K, Ka} whose value on the unknot is $\Ho_{\SU(1)}^\ast(\SU(1)) = \Z[x] \otimes \Lambda(y)$. 
\end{thm}

\begin{proof}
In order to prove this theorem, we will require some detailed structure about the cohomology of Schubert varieties and their Hoshschild homology. These results can be found in the Appendix \ref{Appendix}. By (\cite{Ki1} theorem 8.7), in order to prove the above theorem we need to verify two sets of conditions called the $\mbox{B}$ and $\mbox{M2a/b}$-type conditions given in (\cite{Ki1} definition 8.6). The $\mbox{M2a/b}$-type conditions pertains to the behaviour under the second Markov move and comes down to verifying that the maps given in (\cite{Ki1} theorems 7.1 and 7.2) are both trivial. This follows easily from the behavior of the cohomology of the equivariant spectra $\sBC(w_I)$ under inclusion of subsets $J \subseteq I$ as described in theorem \ref{DomK3-23}. The $\mbox{B}$-type condition is the one flagged in (\cite{Ki1} remark 6.5). We consider it below. 

\medskip
\noindent
In \cite{Ki1}, we considered a certain quasi-equivalence of filtered $\SU(r)$-spectra (details will be given momentarily)
\[ \pi_m : \sBS^{(i,j,2)}(w_I) \longrightarrow \pi_m^\ast \sBS^{(i,j,m+1)}(w_I),\]
where $i,j$ are indices with $1 < m < m_{i,j}$ where $m_{ij}$ are the exponents in the Artin braid group for the Lie group $G$ in question. In the case in hand, $G=\SU(r)$ and so the only nontrivial exponents are $m_{ij} = 3$. 

\medskip
\noindent
Let $Z_m$ denote the fiber of $\pi_m$. Then the relevant $\mbox{B}$-type condition demands that the fiber inclusion map on the associated graded
\[ \Gr_t Z_m \longrightarrow \Gr_t \sBS^{(i,j,m)}(w_I) \]
be surjective in Borel equivariant singular cohomology. The only indices that satisfy the above parameters in the case of $\SU(r)$ are consecutive indices $j=i+1 < r$ and with $m=2$. 

\medskip
\noindent
Now the spectra $\Gr_r \sBS^{(i,i+1,2)}(w_I)$ and $\Gr_r \sBS^{(i,i+1,3)}(w_I)$ are coproducts of other $\SU(r)$-spectra indexed on the same set, and so the relevant condition comes down to verifying a condition on the individual summands. 

\medskip
\noindent
In order to describe the above objects, let us recall the definition of broken Schubert spectra (\cite{Ki1} definition 5.6). 

\medskip
\noindent
Given indices $i, j < r$, let $\Sh_{iji}$ denote the $T \times T$-space given by the pullback diagram
\[
\xymatrix{
\Sh_{iji}   \ar[d] \ar[r]^{\rho_{\Sh}} &  \SU(r) \ar[d] \\
\mathcal{X}_{iji}     \ar[r] & \SU(r)/T, 
}
\]
where $\mathcal{X}_{iji}$ is the image if the following canonical map under group multiplication

\[ \mathcal{X}_{iji} = \mbox{Image of} \, \, \, \,  G_{i} \times_T G_j \times_T G_i/T \longrightarrow \SU(r)/T. \]

\noindent
Notice that $\Sh_{iji}$ is a $T \times T$-invariant subspace of $\SU(r)$, where $T \times T$ acts on $G$ via left/right multiplication. Given any positive indexing sequence of the form $I = \{ i_1, \ldots, i_k, i, j, i \}$, we may construct the spectrum of broken Schubert spectra defined as the suspension spectrum of the space

\[ \BS^{(i,j,3)}(w_I) := \SU(r) \times_T (G_{i_1} \times_T \cdots \times_T G_{i_k} \times_T \Sh_{iji}), \]

\noindent
with the $T$-action on $G_{i_1} \times_T \cdots \times_T G_{i_k} \times_T \Sh_{iji}$ being endpoint conjugation as before. Similarly, we may define the broken Schubert spectra $\BS^{(i,j,2)}(w_I)$, which happen to agree with the spectra of broken symmetries. We now prove a claim that is the heart of the argument that will feed directly into the proof of theorem \ref{DomK2-11c}.

\medskip
\begin{claim} \label{DomK3-13}
Let $i$ and $j$ be indices with $j=i+1 < r$. Given a positive sequence $J = \{ i_1, \ldots, i_k \}$, consider positive indexing sequences
\[ I' = \{i_1, \ldots, i_k, i \}, \quad I'' = \{i_1, \ldots, i_k, i, i \}, \quad I = \{i_1, \ldots, i_{k+3} \} := \{i_1, \ldots, i_k, i,j,i \} \]
so that $J \subset I'$ and $J \subset I'' \subset I$ in the obvious fashion. Then one has a diagram of cofiber sequences of $\SU(r)$-equivariant spectra, which is functorial in $J$ and so that the map $f$ is an equivalence
\[
\xymatrix{
Z_{I''} \ar[r]^{\iota_{I''} \quad} \ar[d]^f & \BC(w_{I''}) \ar[d]^g \ar[r]^{\mu_1} &\BC(w_{I'})  \ar[d] \\
Z_I \ar[r]^{\iota_I \quad } & \BC(w_I) \ar[r]^{\mu \quad} &\BS^{(i,j,3)}(w_I)  
}
\]
where $\mu_1 :  \BC(w_{I''}) \longrightarrow \BC(w_{I'})$ is defined by the multiplication in the last two factors. Furthermore, the maps $\iota_I$ and $\iota_{I''}$ are surjective in Borel equivariant singular cohomology. 
\end{claim}
\begin{proof}
The existence of the commutative square in the right, and its functoriality in $J$ follows from the definition of the spaces in question. Furthermore, by (\cite{Ki1} lemma 5.7), the spaces in the right square form a pushout, and consequently the map $f$ is an equivalence. 

\medskip
\noindent
Now notice that the map $\mu_1$ admits a section and is therefore injective in any cohomology theory. To prove our claim, it is sufficient to show that $\mu$ is also injective in Borel equivariant cohomology. Consider the special case when $J$ is the empty sequence. In that case the map fibers as a $T$-fibration over the map
\[ \overline{\mu} : \BSa(w_{iji}) \longrightarrow \SU(r) \times_T \mathcal{X}_{iji}. \]
In cohomology $\overline{\mu}$ gives rise to a map of $\Ho_T^\ast$-bimodules 
\begin{equation} \label{pushout} \overline{\mu}^\ast : \Ho_{T}(\mathcal{X}_{iji}) \longrightarrow \Ho_{\SU(r)}^\ast(\BSa(w_{iji})), \end{equation}
so that the induced map on Hochschild homology can be identified with $\mu^\ast$. In order to recover $\mu^\ast$ for general sequences $J$, one simply tensors the map $\overline{\mu}^\ast$ on the left with $\Ho_{\SU(r)}^\ast(\BC(w_J))$ before taking Hochschild homology. The upshot of this observation is that in order to show that $\mu^\ast$ is injective in Borel equivariant cohomology, it is sufficient to show that the map (\ref{pushout}) splits as a map of $\Ho_T^\ast$-bimodules. This is what we now show. 

\medskip
\noindent
We begin by observing that there is a pair of pullback diagrams with fiber $G_i/T$
\[
\xymatrix{
 \ESU(r) \times_{\SU(r)} \BSa(w_{iji})   \ar[d] \ar[r]^{\quad \overline{\mu}} & \ESU(r) \times_{T} \mathcal{X}_{iji} \ar[d]^{\pi} \ar[r] & \ESUrT \ar[d] \\
   \ESU(r) \times_{\SU(r)} \BSa(w_{ij})  \ar[r]^{\quad \kappa} &  \ESU(r) \times_{T} \mathcal{Y}_{ij} \ar[r] & \ESUrGi, 
}
\]
where $\mathcal{Y}_{ij}$ are the Schubert varieties in $\SU(r)/G_i$ given by the image of $\mathcal{X}_{iji}$ in $\SU(r)/G_i$. The Eilenberg-Moore spectral sequence for both pullbacks collapses since the top row is free over the bottom in cohomology. In particular, we notice that one gets: 
\[ \Ho_{\SU(r)}^\ast(\BSa(w_{iji})) = \Ho_{\SU(r)}^\ast(\BSa(w_{ij})) \otimes_{\Ho_T^\ast(\mathcal{Y}_{ij} )} \Ho_T^\ast(\mathcal{X}_{iji}). \]
Let us recall \cite{Ku} that the left $\Ho_T^\ast$-module $\Ho_T^\ast(\mathcal{X}_{iji})$ has a Schubert basis indexed by elements in the subgroup isomorphic to $\Sigma_3$ in the Weyl group $\Sigma_r$, generated by reflections $\sigma_i$ and $\sigma_j$. We use $\delta_{i,k}$ and $\delta_{j,k}$ to denote the basis elements in degree $2k$ indexed by the Weyl elements of Bruhat length $k$ that have a reduced expression ending with the reflections $\sigma_i$ and $\sigma_j$ respectively. We will denote $\delta_{i,1}$ and $\delta_{j,1}$ by $\delta_i$ and $\delta_j$ respectively. Similarly, the left $\Ho_T^\ast$-module $\Ho_T^\ast(\mathcal{Y}_{ij})$ has a basis $\delta_{j,k}$ indexed by Weyl elements of length $k$ that have a reduced expression ending with $\sigma_j$. These basis elements are compatible under $\pi^\ast$. 

\medskip
\noindent
In order to construct a splitting to $\overline{\mu}^\ast$ as bimodules, it is sufficient to construct a splitting of the map $\kappa$ in cohomology, in the category of modules over $\Ho_T^\ast(\mathcal{Y}_{ij})$.

\medskip
\noindent
Notice that we have a lift to $\kappa$ induced by inclusions of Schubert varieties
\[ \tilde{\kappa} :   \ESU(r) \times_{\SU(r)} \BSa(w_{ij}) \longrightarrow  \ESU(r) \times_{T} \mathcal{X}_{iji}. \]
It will therefore be more convenient for us to work in the cohomology ring $\Ho_T^\ast(\mathcal{X}_{iji})$ in order to make computations before restricting to $\Ho_{\SU(r)}^\ast(\BSa(w_{ij}))$ along this section. 

\medskip
\noindent
Since $\mathcal{X}_{iji}$ is the full flag variety for $\SU(3)$, standard formulas for Schubert multiplication, (see \cite[11.3.17]{Ku}) show that exists an element $c \in \Ho_T^2$ that satisfies the following relation among the Schubert basis elements: 
\[ \delta_j  \delta_i =  \delta_{i, 2} + \delta_{j, 2}, \quad \quad \delta_j \delta_j = \delta_{j, 2} + c \delta_j.  \]
Now consider the class $e \in \Ho_T^2(\mathcal{X}_{iji})$ given by: 
\[ e := \delta_i - \delta_j + c. \] 
It is easy to see that the restriction of $e$ to $\Ho_{\SU(r)}^\ast(\BSa(w_{ij}))$ along $\overline{\mu}$ generates a free $\Ho_T^\ast$-submodule complementary to the image of $\Ho_T^\ast(\mathcal{Y}_{ij})$. It remains to show that this sub-module is closed under multiplication with $\Ho_T^\ast(\mathcal{Y}_{ij})$. 

\smallskip
\noindent
Now, by using the multiplication rules for the Schubert basis, we observe that $e \delta_j = \delta_{i, 2} $. In particular, $e \delta_j$ is zero in $\Ho_{\SU(r)}^\ast(\BSa(w_{ij}))$, since $\delta_{i,2}$ restricts to zero along $\tilde{\kappa}$. Similarly, recalling the relation $\delta_j \delta_j = \delta_{j, 2} + c \delta_j$, it also follows that $e \delta_{j, 2}$ restricts to zero. The above argument shows that the $\Ho_T^\ast(\mathcal{Y}_{ij})$-submodule in $\Ho_{\SU(r)}^\ast(\BSa(w_{ij}))$ generated by the restriction of the class $e$ is a summand that is a complement to $\Ho_T^\ast(\mathcal{Y}_{ij})$. This is what we wanted to prove. 
\end{proof}

\medskip
\noindent
Let us now complete the proof of \ref{DomK2-11c}. As mentioned earlier, the proof follows once we have verified the $\mbox{B}$-type condition of (\cite{Ki1} definition 8.6) which comes down to showing that the map in (\cite{Ki1} claim 6.5) is injective. This is equivalent to claim \ref{DomK3-13} above with the cosmetic difference that one must allow the subsequence $\{i, j, i \}$ anywhere in the sequence $I$ and not just at the terminating three spots. This is not a problem since we can always invoke the first Markov property to move the subsequence to the end. 
\end{proof}

\section{$\LSU(r)$-equivariant lifts of Strict Broken Symmetries and Twistings} \label{LiftsSBS}

\medskip
\noindent
In this section we would like to describe the framework that allows us to twist equivariant cohomology theories by local systems on $\SU(r)$. In order to make this statement precise, we will first construct $\LSU(r)$-equivariant lifts of our spaces of broken symmetries. 

\medskip
\noindent
Consider the homomorphism that evaluates a loop at the point $1 \in S^1$. 
\[ \Ev : \LSU(r) \longrightarrow \SU(r), \quad \Ev(\varphi) = \varphi(1). \]
The homomorphism $\Ev$ allows us to descend from $\LSU(r)$-spaces to $\SU(r)$-spaces by induction. Namely, given a pointed $\LSU(r)$-space $Y$, we may descend to an $\SU(r)$-space $Y_{\Ev}$ defined as $Y_{\Ev} := Y \wedge_{\LSU(r)} \SU(r)_+$. Applying this construction to the $\LSU(r)$-space $A_r$ of principal $\SU(r)$-connections on the trivial $\SU(r)$-bundle over the circle $S^1$, we notice that the induced map may be identified with the holonomy map
\[ \mbox{Hol} : A_r \longrightarrow (A_r)_{\Ev} \cong \SU(r), \]
with the induced $\SU(r)$-action on $\SU(r)$ being the conjugation action. In fact, the holonomy map above is a principal $\OSU(r)$-bundle and gives rise to an equivalence of stacks between $A_r \qu \LSU(r)$ and $\SU(r) \qu \SU(r)$. 

\medskip
\noindent
The above example also suggests a way one may reverse the procedure by starting with an $\SU(r)$-space $X$ endowed with an equivariant map, with $\SU(r)$ acting on itself by conjugation
\[ \rho_X : X \longrightarrow \SU(r), \]
and define $X^{\Ev}$ by pulling back the holonomy map along $\rho_X$. Since $\LSU(r)$ is generated by the groups $\SU(r)$ and $\OSU(r)$, both of which act on $X^{\Ev}$ (by the naturality of the pullback construction), we see that $X^{\Ev}$ is an $\LSU(r)$-space.

\medskip
\noindent
We now describe a proper $\LSU(r)$-action on a space $\gA_r$ that is a smaller model for the space $A_r$ of principal $\SU(r)$-connections on a circle. First let $\gSA_r$ denote the space of principal $\SSU(r)$-connections on the circle. Let $\ZSU(r)$ denote the center of $\SU(r)$. Now notice that the canonical map $\SSU(r) \times \ZSU(r) \longrightarrow \SU(r)$ is an $r$-fold cover. We may therefore define the space underlying $\gA_r$ to be 
\[ \gA_r := \gSA_r \times \R. \]
We would now like to endow $\gA_r$ with a proper action of $\LSU(r)$. We first choose an appropriate model for $\LSU(r)$ of the form $\LSU(r) = \LSSU(r) \rtimes \LSU(1)$, where $\LSSU(r)$ is the subgroup of smooth loops with values in $\SSU(r)$, and for $\LSU(1)$ we take the smaller subgroup of 
Laurent polynomials with values in $\SU(1)$. 

\bigskip
\begin{claim} \label{DomK2-5}
With the above model for $\LSU(r)$, the action of $\LSSU(r)$ on $\gA_r = \gSA_r \times \R$ extends to an action of $\LSU(r)$. 
\end{claim}

\begin{proof}
To extend the action of $\LSSU(r)$ on $\gA_r$ to an action of $\LSU(r)$, we simply need to describe an action of $\LSU(1)$ that is compatible with the action of $\LSSU(r)$. Recall that we have chose a model so that $\LSU(1) = \SU(1) \times \Z$, where $\Z = \pi_1(\SU(1))$. 

\medskip
\noindent
Let us begin by fixing an isomorphism between $\SU(1)$ and a subtorus $\mbox{T}_r \subseteq \SSU(r) \cap T^r$ so that $\mbox{T}_r$ contains the center of $\SSU(r)$ (the final answer will be equivalent for all choices of $\mbox{T}_r$). As such, we may define the action of an element $z \in \SU(1)$ on $(\nabla, x) \in \gSA_r \times \R$ by
\[ z \ast (\nabla, x) = (\mbox{Ad}_{z^{\frac{1}{r}}}(\nabla), x). \]
It is easy to check that this action is well defined. Similarly, given a generator $\sigma \in \pi_1(\SU(1))$ seen as an element in $\LSU(1)$, the action of $\sigma$ on the element $(\nabla, x)$ is given by
\[ \sigma \ast (\nabla, x) = \big{(}\mbox{Ad}_{\sigma^{\frac{1}{r}}}(\nabla) + \frac{1}{r}d\theta, \frac{1}{r} + x\big{)}, \]
where $\theta$ represents the fundamental one-form on the circle with values in the Lie algebra of the subtorus $\mbox{T}_r \subseteq \SSU(r)$ and $\sigma^{\frac{1}{r}}$ denotes any pointwise $r$-th root of $\sigma \in \LSU(1)$. As before, it is easy to check that this formula is well defined and that it describes an action of our chosen model for $\LSU(r)$ extending the action of $\LSSU(r)$. 
\end{proof}

\medskip
\begin{claim} \label{DomK2-7}
The space $\gA_r$ is a proper $\LSU(r)$-space with a free $\OSU(r)$-action. Furthermore, the induced $\SU(r)$-space ${\gA_r}_+ \wedge_{\LSU(r)} \SU(r)_+$ is equivalent to the conjugation action of $\SU(r)$ on itself. 
\end{claim}

\begin{proof}
For any simply connected, compact Lie group $G$, it is well-known \cite{KM} that the space of principal $G$-connections $A$ over a circle is a proper $\LG$-space, where $\LgG$ denotes the Loop group of $G$. The pointed gauge group $\Omega(G)$ is known to act freely on the space of connections $A$. Furthermore, the holonomy map establishes an equivalence between the induced space $A_+ \wedge_{\LgG} G_+$ and the conjugation action of $G$ on itself. Applying this to our example $\gA_r = \gSA_r \times \R$, it follows that the pointed loop group $\OSU(r) := \OSSU(r) \rtimes \Z$ acts freely on $\gA_r$, with the orbit space being $\SSU(r) \times_{\Z/r\Z} (\R/\Z)$. Identifying this orbit space with $\SU(r)$, it is straigforward to see that the residual action of $\SU(r)$ on this orbit space is equivalent to the conjugation action. Hence, the stack $\gA_r \qu \LSU(r)$ is a proper $\LSU(r)$-space equivalent to $\SU(r) \qu \SU(r)$. 
\end{proof}

\medskip
\noindent
We consolidate the previous claims into the following definition

\medskip
\begin{defn} (The universal proper $\LSU(r)$-space $\gA_r$) \label{DomK2-8}

\noindent
Taking the small model for $\LSU(r)$, we define the universal proper $\LSU(r)$-space to be the space $\gA_r$. The subgroup $\OSU(r) := \OSSU(r) \rtimes \Z$ acts freely on $\gA_r$ so that there is principal $\OSU(r)$-fibration defined as the ``holonomy' map 
\[ \mbox{Hol} : \gA_r \longrightarrow \SU(r). \]
Furthermore, the induced space ${\gA_r}_+ \wedge_{\LSU(r)} \SU(r)_+$ is equivalent to the conjugation action of $\SU(r)$ on itself. 
\end{defn}

\medskip
\begin{remark}
As the definition suggests, it is in fact true that $\gA_r$ is the terminal proper $\LSU(r)$-space (up to equivariant homotopy), even though we don't really need that fact. 
\end{remark}

\medskip
\noindent
Now let $I = \{ i_1, i_2, \ldots, i_k \}$ denote an indexing sequence with $i_j < r$. Let $T$ denote the standard maximal torus, and let $G_i$ denote the unitary (block-diagonal) form in the reductive Levi subgroup having roots $\pm \alpha_i$. We consider $G_i$ as a two-sided $T$-space under the canonical left(resp. right) multiplication. For the (positive) braid word $w_I$, recall the spaces $\BC(w_I)$ of broken symmetries  
\[ \BC(w_I) := \SU(r) \times_T ({G_{i_1}} \times_T {G_{i_2}} \times_T \cdots \times_T {G_{i_k}}) = \SU(r) \times_T \BC_T(w_I) , \]
with the $T$-action on $\BC_T(w_I) := (G_{i_1} \times_T G_{i_2} \times_T \cdots \times_T G_{i_k})$ given by endpoint conjugation
\[ t \, [(g_1, g_2, \cdots, g_{k-1}, g_k)] := [(tg_1, g_2, \cdots, g_{k-1}, g_kt^{-1})]. \]

\medskip
\begin{defn} (Lifts of the spectra $\BC(w_I)$ to $\LSU(r)$-spectra) \label{Twist-2}

\noindent
Given a positive braid word $w_I$, let $\rho_I$ denote the canonical $\SU(r)$-equivariant map on $\BC(w_I)$ induced by group multiplication on the factors and with values in the space $\SU(r)$ acting on itself by conjugation 
\[ \rho_I :  \BC(w_I) = \SU(r) \times_T ({G_{i_1}} \times_T {G_{i_2}} \times_T \cdots \times_T {G_{i_k}}) \longrightarrow \SU(r), \quad [(g, g_{i_1}, \ldots ,g_{i_k})] \mapsto gg_{i_1}\ldots g_{i_k}g^{-1}. \]
We define the space $\BC(w_I)^{\Ev}$ to be the pullback of the universal proper $\LSU(r)$-space $\gA_r$ along $\rho_I$, and denote $\mbox{Hol}_I : \BC(w_I)^{\Ev} \longrightarrow \BC(w_I)$ to be the induced holonomy map. Note that we may identify $\BC(w_I)$ with $\BC(w_I)^{\Ev}_+ \wedge_{\LSU(r)} \SU(r)_+$. 

\medskip
\noindent
Given an arbitrary indexing sequence $I = \{ \epsilon_{i_1} i_1, \cdots, \epsilon_{i_k} i_k \}$, we define the $\LSU(r)$-equivariant spectrum of broken symmetries $\BC(w_I)^{\Ev}$ as the Thom spectrum of the pullback of $-\zeta_I$ (section \ref{DomK2})
\[ \BC(w_I)^{\Ev} := (\BC(w_{\underline{I}})^{\Ev})^{-\zeta_I}, \]
where $\underline{I}$ is the indexing sequence obtained from $I$ by replacing each $\epsilon_j$ with $1$. The pullback is performed along $\mbox{Hol}_{\underline{I}} : \BC(w_{\underline{I}})^{\Ev} \longrightarrow \BC(w_{\underline{I}})$
\end{defn}

\medskip
\noindent
In order to extend the above definition to the filtered spectrum of strict broken symmetries, we will need to reconsider a $\LSU(r)$-equivariant variant of the Pontrjagin-Thom construction (\cite{Ki1} claim 2.5).  

\medskip
\begin{claim} \label{DomK3-3}
If the set $J$ is obtained from $I$ by dropping an entry $-i_j$, then the Pontrjagin-Thom construction induces a canonical $\LSU(r)$-equivariant map
\[ \pi_{i_j}^{\Ev} : \BC(w_I)^{\Ev} \longrightarrow \BC(w_J)^{\Ev}. \] 
\end{claim}
\begin{proof}
The Pontrjagin-Thom construction that gives rise to the map $\pi_{i_j}^{\Ev}$ is performed at the level of the spaces $\BC(w_{\underline{I}})^{\Ev}$. In order to study this, recall the $\SU(r)$-equivariant principal $\OSU(r)$-bundle
\[ \mbox{Hol} : \gA_r \longrightarrow \SU(r). \]
One may pick a $\SU(r)$-invariant connection $\nabla^{\Ev}$ on the bundle $\mbox{Hol}$ by first taking any connection and then averaging over $\SU(r)$ to get an invariant connection (in fact, there is even a canonical choice for such a connection). The connection $\nabla^{\Ev}$ induces compatible connections $\nabla_{\underline{I}}^{\Ev}$ on all $\SU(r)$-equivariant principal $\OSU(r)$-bundles $\mbox{Hol}_{\underline{I}} : \BC(w_{\underline{I}})^{\Ev} \longrightarrow \BC(w_{\underline{I}})$. 

\medskip
\noindent
Now consider the special case of the above claim where $I$ has a single entry $-i_j$ for which $\epsilon_{i_j} = -1$, and so that $J$ is the set obtained by dropping $-i_j$. It is straightforward to see that the restriction of $\zeta_{i_j}$ to $\BC(w_J)$ is the normal bundle of the canonical inclusion $\BC(w_{J}) \subset \BC(w_{\underline{I}})$. Let $\eta_J$ denote the $\SU(r)$-equivariant tubular neighborhood of $\BC(w_J)$ in $\BC(w_{\underline{I}})$ identified with $\zeta_{i_j}$ via the exponential map. Parallel transport via the connection $\nabla_{\underline{I}}^{\Ev}$ along linear paths in in the neighborhood $\zeta_{i_j}$ allows us to canonically identify a $\LSU(r)$-equivariant tubular neighborhood of the inclusion $\BC(w_J)^{\Ev} \subset \BC(w_{\underline{I}})^{\Ev}$ with the pullback of $\zeta_{i_j}$ along $\mbox{Hol}_J$. We may therefore perform the Pontrjagin-Thom construction to get a map $\BC(w_{\underline{I}})^{\Ev} \longrightarrow (\BC(w_J)^{\Ev})^{\zeta_{i_j}}$. 
Twisting with $-\zeta_{i_j}$ gives rise to the expected map
\[ \pi^{\Ev}_{i_j} : \BC(w_I)^{\Ev} \longrightarrow \BC(w_J)^{\Ev}. \]

\noindent
From this special case, it is straightforward to deduce the general case since the remaining bundles do not interfere with the Pontrjagin-Thom construction for $i_j$. 
\end{proof}

\bigskip
\noindent
\begin{defn} (The functor $\BC(w_I)^{\Ev}$, compare \cite{Ki1} defintion 2.6) \label{DomK3-4}

\noindent
Given a braid word $w_{I}$, for $I = \{ \epsilon_{i_1} i_1, \cdots, \epsilon_{i_k} i_k \}$, let $2^{I}$ denote the set of all subsets of $I$. Let us define a poset structure on $2^{I}$ generated by demanding that nontrivial indecomposable morphisms $J \rightarrow K$ have the form where either $J$ is obtained from $K$ by dropping an entry $i_j \in K$ (i.e. an entry for which $\epsilon_{i_j} = 1$), or that $K$ is obtained from $J$ by dropping an entry $-i_j$ (i.e an entry for which $\epsilon_{i_j} = -1$). 

\medskip
\noindent
The construction $\BC(w_{J}^{\Ev})$ induces a functor from the category $2^{I}$ to $\LSU(r)$-spectra. More precisely, given a nontrivial indecomposable morphism $J \rightarrow K$ obtained by dropping $-i_j$ from $J$, the induced map $\BC(w_{J})^{\Ev} \rightarrow \BC(w_{K})^{\Ev}$ is obtained by applying the map $\pi^{\Ev}_{i_j}$ of claim \ref{DomK3-3}. Likewise, if $J$ is obtained from $K$ by dropping the factor $i_j$, then the map $\BC(w_{J})^{\Ev} \rightarrow \BC(w_{K})^{\Ev}$ is defined as the canonical inclusion. 
\end{defn}

\bigskip
\begin{defn} \label{DomK3-5} ($\LSU(r)$-equivariant strict broken symmetries, compare \cite{Ki1} definitions 2.7, 2.8) 

\noindent
let $I^+ \subseteq I$ denote the terminal object of $2^{I}$ given by dropping all terms $-i_j$ from $I$ (i.e terms for which $\epsilon_{i_j} = -1$). Define the poset category $\I$ to the subcategory of $2^{I}$ given by removing $I^+$. 
\[ \I = \{ J \in 2^{I}, \,  J \neq I^+ \}. \]

\noindent
We first define the equivariant $\LSU(r)$-spectrum $\sBCU(w_{I})^{\Ev}$ via the cofiber sequence of equivariant $\LSU(r)$-spectra
\[ \hocolim_{J \in \I} \BC(w_{J})^{\Ev} \longrightarrow \BC(w_{I^+})^{\Ev} \longrightarrow \sBCU(w_{I})^{\Ev}. \]
We endow $\sBCU(w_{I})^{\Ev}$ with a natural filtration as $\LSU(r)$-spectra giving rise to the filtered spectrum of strict broken symmetries $\sBC(w_I)^{\Ev}$ as follows. The lowest filtration is defined as
\[ F_0 \, \sBC(w_{I})^{\Ev} = \BC(w_{I^+})^{\Ev}, \quad \mbox{and} \quad F_k \, \sBC(w_{I})^{\Ev} = \ast, \quad \mbox{for} \quad k < 0. \]
Higher filtrations $F_t$ for $t>0$ are defined as the cone on the restriction of $\pi$ to the subcategory $\I^t \subseteq \I$ consisting of objects no more than $t$ nontrivial composable morphisms away from $I^+$. In other words $F_t \, \sBC(w_{I})^{\Ev}$ is defined via the cofiber sequence
\[ \hocolim_{J \in \I^t} \BC(w_{J})^{\Ev} \longrightarrow \BC(w_{I^+})^{\Ev} \longrightarrow F_t \, \sBC(w_{I})^{\Ev}. \]
\end{defn}

\medskip
\begin{remark} \label{DomK3-6}
As before, it is straightforward to see that the associated graded of this filtration is given by 
\[ \Gr_t(\sBC(w_J)^{\Ev}) =  \bigvee_{{J} \in \I^t/\I^{t-1} } \BC(w_{J})^{\Ev}. \]
\end{remark}

\noindent
Having constructed the filtered $\LSU(r)$-equivariant homotopy type $\sBC(w_J)^{\Ev}$, we now proceed to show that this homotopy type is independent of the presentation $w_I$, up to quasi-equivalence, and can be normalized to give rise to an invariant of the link $L$ obtained by closing the braid $w_I$. The proof of this fact follows formally from the proofs given in sections 4, 5, 6 and 7 of \cite{Ki1}. One simply invokes the following technical claim that allow one to lift properties of broken symmetries over to their $\LSU(r)$-lifts. 

\medskip
\begin{claim} \label{DomK3-8}
Let $P$ be a pushout of $\SU(r)$-spaces $X$, $Y$ and $Z$ over the $\SU(r)$-space $\SU(r)$ as follows
\[
\xymatrix{
X \ar[d]^{g} \ar[r]^{h} & Z \ar[d] \\
Y    \ar[r]  &  P.
}
\]
Then the above diagram lifts to a pushout diagram of $\LSU(r)$-spaces over $\gA_r$
\[
\xymatrix{
X^{\Ev} \ar[d]^{g^{\Ev}} \ar[r]^{h^{\Ev}} & Z^{\Ev} \ar[d]  \\
Y^{\Ev}    \ar[r]  &  P^{\Ev}  }
\]
\end{claim}
\begin{proof}
By definition, we express $P$ as the quotient space of $Y \coprod Z$ under the relations indexed by $X$ that identify $g(x)$ with $h(x)$ for any point $x \in X$. Now notice that $P^{\Ev}$ is a principal $\OSU(r)$-bundle over $P$. As such we may express $P^{\Ev}$ as the quotient space of the induced $\OSU(r)$-bundles over $Y$ and $Z$, with identifications indexed by the induced bundle over $X$. This is precisely the content of the claim.
\end{proof}

\medskip
\noindent
It is now a simple matter of going through the statements in sections 4, 5, 6 and 7 in \cite{Ki1} sequentially to show why the same statements formally hold for the $\LSU(r)$-equivariant lifts. We sketch the details for the benefit of the interested reader. 

\medskip
\noindent
Let us begin by addressing why $\BC(w_I)^{\Ev}$ is independent of the presentation $w_I$. This involves two properties. Namely, invariance under the braid relations and invariance under the inverse relation. We start by indicating why $\sBC(w_I)^{\Ev}$ is invariant under the braid relations. 

\medskip
\noindent
In (\cite{Ki1} theorem 5.1), given a pair of indices $(i,j)$, we considered an indexing sequence $I^{(i,j)}$ which contains a subsequence of consecutive terms given by the braid sequence $\{ i, j, i, j, \ldots \}$ with $m_{i,j}$-terms. Then the filtered $\SU(r)$-spectrum of strict broken symmetries $\sBC(w_{I^{(i,j)}})$ was connected to the spectrum $\sBC(w_{I^{(j,i)}})$ by a zig-zag of elementary quasi-equivalences, where $I^{(j,i)}$ represents the same sequence with the subsequence $\{ i, j, i, j, \ldots \}$ replaced by the sequence $\{ j, i, j, i, \ldots \}$ with the same number of terms. The method of proof entailed constructing a sequence of filtered $\SU(r)$-spectra (\cite{Ki1} definition 5.6) known as strict broken Schubert spectra $\sBS^{(i,j,m)}(w_I)$ for integers $1 \leq m \leq m_{i,j}$. These filtered spectra were constructed as homotopy colimits of functors $\BS^{(i,j,m)}(w_J)$ over certain quotient poset categories. Furthermore, we had $\sBS^{(i,j,1)}(w_I) = \sBC(w_{I^{(i,j)}})$ (and the same for $(j,i)$), and that the terminal filtered spectra $\sBS^{(i,j,m_{i,j})}(w_I)$ and $\sBS^{(j,i,m_{i,j})}(w_I)$ at the end of each sequence agreed. In addition, we showed that all the filtered $\SU(r)$-spectra $\sBS^{(i,j,m)}(w_I)$ were connected by zig-zags of elementary quasi-equivalences. This quasi-equivalence was constructed by mean of a comparison map between broken Schubert spectra $\BS^{(i,j,m)}(w_J) \longrightarrow \BS^{(i,j,m+1)}(w_J)$ for subsets $J \subseteq I^{(i,j)}$ and $1 \leq m < m_{i,j}$. 

\bigskip
\noindent
Let us now proceed to demonstrate how the above argument extends to the $\LSU(r)$-lifts of these filtered spectra.

\bigskip
\noindent
Let $\underline{J}$ denote the set obtained by replacing all $\epsilon_j$ by $1$. From the definition of Schubert spectra, it is straightforward to check that the canonical map $\rho_J : \BC(w_{\underline{J}}) \longrightarrow \SU(r)$ of \ref{Twist-2} factors through the spectra $\BS^{(i,j,m)}(w_{\underline{J}})$ for all $1 < m < m_{i,j}$.
\[
\xymatrix{
\BC(w_{\underline{J}}) \ar[d]^{\rho_{\underline{J}}} \ar[r] & \BS^{(i,j,m)}(w_{\underline{J}})  \ar[d]^{\rho_{\underline{J}}^{(i,j,m)}} \ar[r]& \BS^{(i,j,m+1)}(w_{\underline{J}}) \ar[d]^{\rho^{(i,j,m+1)}_{\underline{J}}} \\
\SU(r) \ar[r]^{=} & \SU(r)  \ar[r]^{=} & \SU(r).
}
\]
It follows that the comparison maps lift to yield comparison maps of equivariant spectra
\begin{equation} \label{DomK3-9} 
 \BC^{(i,j,m)}(w_J)^{\Ev} \longrightarrow \BS^{(i,j,m+1)}(w_J)^{\Ev}
\end{equation}

\medskip
\noindent
The next step in the proof of braid invariance is to show that the fiber of the comparison map induces a zig-zag of maps of filtered spectra, with acyclic fibers. The key point in showing acyclicity comes down to showing that the fibers of a particular pair of maps of the form given in equation (\ref{DomK3-9}) are equivalent. These maps fit into a pushout diagram before taking the $\LSU(r)$-equivariant lift and so we may invoke claim \ref{DomK3-8} to observe that the equivalence of fibers remains true on taking the lift. The rest of the argument is formal. 

\bigskip
\noindent
We now move to invariance under the inverse relation. This relation involves showing that for any index $i$ and a sequence $I$ of the form $\{ \ldots, i, -i, \ldots \}$, the filtered $\LSU(r)$-spectrum $\BC(w_I)^{\Ev}$ admits an elementary quasi-equivalence with $\BC(w_I')^{\Ev}$, where $I'$ is obtained from $I$ by dropping the pair $\{ i, -i \}$. Of course, we require invariance under the pair $\{-i, i\}$ as well. 

\medskip
\noindent
The key construction in the argument given in (\cite{Ki1} section 6) involves splitting the two maps
\[ \BC(w_{I_1}) \longrightarrow \BC(w_I), \quad \BC(w_I) \longrightarrow \BC(w_{I_2}), \quad \mbox{where} \quad I_1 = I/\{i\}, \quad I_2 = I/\{-i\}. \]
It is straightforward to check that the splitting described in (\cite{Ki1} claims 6.2 and 6.3) are maps over $\SU(r)$, and can therefore be lifted to splittings of the maps
\[ \BC(w_{I_1})^{\Ev} \longrightarrow \BC(w_I)^{\Ev}, \quad \BC(w_I)^{\Ev} \longrightarrow \BC(w_{I_2})^{\Ev}. \]
Again, the rest of the argument is purely formal.

\bigskip
\noindent
Being done with showing that $\sBC(w_I)^{\Ev}$ is independent of presentation, let us now indicate why $\sBC(w_I)^{\Ev}$ is an invariant of links. This requires showing invariance under first Markov property, and the second Markov property. The $\LSU(r)$-equivariant versions of these are straightforward, given the $\SU(r)$-equivariant versions shown in (\cite{Ki1} section 4 and 7). The map $\tau$ (\cite{Ki1} theorem 4.2) that establishes invariance under the first Markov property is easily seen to be a map over the space $\SU(r)$ and therefore lifts to the $\LSU(r)$-equivariant strict broken symmetries. 

\medskip
\noindent
Similarly, for invariance under the second Markov property, one simply observes that the relevant cofibration sequences described in (\cite{Ki1} claim 7.3 and 7.5) admit canonical lifts. The above discussion allows us to define invariants of braids and links resp. $\sBC(w)^{\Ev}$ and $\sBC(L)^{\Ev}$ as in (\cite{Ki1} definitions 2.10 and 8.4). 

\bigskip
\begin{thm} \label{DomK3-12} 
Let $L$ denote a link described by the closure of a positive braid $w \in \Br(r)$ with $r$-strands, and let $w_I$ be a presentation of $w$ as $w_I = \sigma^{\epsilon_1}_{i_1} \ldots \sigma^{\epsilon_k}_{i_k}$, with $I = \{ \epsilon_1 i_1, \ldots, \epsilon_k i_k \}$. Let the $\LSU(r)$-equivariant normalized spectrum of strict broken symmetries $\sBC(L)^{\Ev}$ be defined as 
\[ \sNBC(L)^{\Ev} := \Sigma^{l(w_I)} \sBC(w_I)^{\Ev}[\varrho_I], \]
where $\Sigma^{l(w_I)}$ denotes the suspension by $l(w_I) := l_-(w_I) - 2l_+(w_I)$ with $l_+(w_I)$ being the number of positive and $l_-(w_I)$ being the number of negative exponents in the presentation $w_I$ for $w$ in terms of the generators $\sigma_i$. Also, $\sBC(w_I)^{\Ev}[\varrho_I]$ denotes the filtered spectrum $\sBC(w_I)^{\Ev}$ with a shift in indexing given by $F_t \, \sBC(w_I)^{\Ev}[\varrho_I] := F_{t+\varrho_I} \, \sBC(w_I)^{\Ev}$, with $\varrho_I$ being one-half the difference between the cardinality of the set $I$, $|I|$, and the minimum word length $|w|$, of $w \in \Br(r)$
\[ \varrho_I = \frac{1}{2} (|I| - |w|). \]
Then, as a function of the (framed) link $L$, the filtered $\LSU(r)$-spectrum $\sBC(L)^{\Ev}$ is well-defined up to quasi-equivalence. See \cite{Ki1} Section 8, for a discussion on the framed enhancement.
\end{thm}

\smallskip
\begin{remark} \label{DomK3-12a}
In (\cite{Ki1} theorem 6.8) we showed that $\sBC(L)$ was equivalent to $\sBC(L^R)$, where $R$ was a reflection symmetry that was induced by reversing the order of braids. This symmetry also lifts to a levelwise (honest) equivalence between $\sBC(L)^{\Ev}$ and $\sBC(L^R)^{\Ev}$ and corresponds to the automorphism of the space of connections induced by complex conjugation acting on $S^1$. 
\end{remark}

\smallskip
\begin{remark} \label{Galois4}
The holomomy maps $\rho_I$ are equivariant with respect to the Galois symmetry $\sigma$ defined in \cite{Ki1} (Section 8). In particular, $\sigma$ can be shown to lift to a Galois symmetry on the $\LSU(r)$-equivariant spectrum $\sBC(L)^{\Ev}$, with $\sigma$ acting on $\LSU(r)$ by pointwise complex conjugation. 
The action of $\sigma$ on the normalization factor $\Sigma^{l(w_I)}$ is as described in \cite{Ki1} (Remark 8.7). 
\end{remark}

\bigskip
\noindent
We now study how one might use the $\LSU(r)$-equivariant lift of strict broken symmetries to compute link invariants by applying suitable twisted equivariant cohomology theories. Towards this end, assume that $\E_{\SU(r)}$ is a family of $\SU(r)$-equivariant cohomology theories as studied in \cite{Ki1}. In particular, we assume natural compatiblity under restrictions 
\[ \Delta_j : \E_{\SU(r_1)} \wedge \E_{\SU(r_k)} \ldots \wedge \E_{\SU(r_m)} \llra{\cong} j^\ast \E_{\SU(m)}, \, \, \mbox{where} \, \, \,  j : \SU(r_1) \times \SU(r_2) \times \ldots \times \SU(r_k) \hookrightarrow \SU(m), \]
is an arbitrary injection of Lie groups. We now impose further requirements that $\E_{\SU(r)}$ is associative (i.e. $A_{\infty}$), and homotopy commutative with this multiplicative structure being preserved under the above equivalences.  We may take for instance, $\E_{\SU(r)}$ to be the restriction of a multiplicative Global cohomology theory \cite{Sc}. 

\medskip
\noindent
Now let $\Delta_{r,s}$ denote the equivalence induced by the diagonal $\iota_{r,s} : \SU(r) \times \SU(s) \longrightarrow \SU(t)$ for all pairs $(r,s)$ with $r+s = t$
\[  \Delta_{r,s} : \E_{\SU(r)} \wedge \E_{\SU(s)} \llra{\cong} \iota_{r,s}^\ast\E_{\SU(r+s)}. \]

\medskip
\begin{defn} (Twistings for the family of equivariant cohomology theories $\E$) \label{Twist-6a}

\noindent
Given a family of sufficiently multiplicative cohomology theories $\E$ as above, a twisting $\theta$ for $\E$ is a family $\theta_r$ of $\SU(r)$-equivariant action maps in the category of right $\E_{\SU(r)}$-module spectra
\[ \theta_r  :\OSU(r)_+ \wedge \E_{\SU(r)} \longrightarrow \E_{\SU(r)}, \]
where $\OSU(r)$ is a $\SU(r)$-space under conjugation. We demand that the action maps are $\SU(r) \times \SU(s)$-equivarianty compatible for all pairs $r+s=t$
\[
\xymatrix{
  (\OSU(r) \times \OSU(s))_+ \wedge \E_{\SU(r)} \wedge \E_{\SU(s)} \ar[r]^{\quad \quad \quad \iota \wedge \Delta} \ar[d]^{\theta_r \wedge \theta_s} & \OSU(t)_+ \wedge \iota_{r,s}^\ast \E_{\SU(t)} \ar[d]^{\theta_t}    \\
  \E_{\SU(r)} \wedge \E_{\SU(s)} \ar[r]^{\quad \quad \quad \Delta}  & \iota_{r,s}^ \ast \E_{\SU(t)}). 
}
\]
\end{defn}

\medskip
\noindent
In particular, a twist $\theta$ extends a $\SU(r)$-spectrum $\E_{\SU(r)}$ to a $\LSU(r)$-spectrum:

\medskip
\begin{defn} ($\LSU(r)$-equivariant spectra and $\LSU(r)$-equivariant cohomology theories) 

\noindent
An $\LSU(r)$-equivariant spectrum is a collection of $\LSU(r)$-spaces $\mbox{E}(V)$ indexed on complex representations $V$ in a complete $G$-universe for the group $G = \SU(r)$. Notice that these representations may be considered as representations of the loop groups $\LSU(r)$ via the evaluation map $\Ev : \LSU(r) \longrightarrow \SU(r)$. We require these spaces to be related by compatible collection of $\LSU(r)$-equivariant maps $ \mbox{E}(W) \longrightarrow \Omega^V \mbox{E}(W\oplus V)$ that are $\SU(r)$-equivariant equivalences. 

\medskip
\noindent
Even though we will be primarily interested in integer graded cohomology, an $\LSU(r)$-equivariant spectrum $\mbox{E}$ represents a $\SU(r)$-representation graded cohomology theory in the usual way: Given any $\LSU(r)$-equivariant spectrum $Z$, the value of the $\mbox{E}$-cohomology of $Z$ in grading given by a $\SU(r)$-representation $V$ is defined by 
\[ \mbox{E}^V(Z) := \pi_0 \Map^{\LSU(r)}(Z, \mbox{E}(V)). \]
We observe that a twisting $\theta$ on a family $\{ \E_{\SU(r)} \}$ of $\SU(r)$-equivariant cohomology theories allows for an extension of $\{ \E_{\SU(r)} \}$ to $\LSU(r)$-equivariant cohomology theories. 
\end{defn}

\medskip
\begin{remark} \label{units}
Given a twist $\theta$ as in definition \ref{Twist-6a}, consider the subspace $\GL(\E_{\SU(r)}) \subset \Omega^{\infty} \E_{\SU(r)}$ consisting of components whose underlying element in the ring $\pi_0(\E_\natural)$ is a unit, where $\E_\natural$ denotes the spectrum $\E_{\SU(r)}$ with the $\SU(r)$-action forgotten. $\GL(\E_{\SU(r)})$ is a group-like monoid with a $\SU(r)$-action. We see that $\theta_r$ restricts to a $\SU(r)$-equivariant map of monoids $\OSU(r) \longrightarrow \GL(\E_{\SU(r)})$ which we can classify by a $\SU(r)$-equivariant map (denoted by the same name)
\[ \theta_r : \SU(r) \longrightarrow \BGL(\E_{\SU(r)}). \]
Such a classification would allow us to recover the twist if we knew that the group completion map
\[ \GL(\E_{\SU(r)}) \longrightarrow \Omega \BGL(\E_{\SU(r)}) \] 
 was a $\SU(r)$-equivariant weak equivalence. Such an equivalence is not true in general, but would indeed be true if we were to make a technical assumption that for any compact subgroup $H \leq \SU(r)$, a homotopy class $x \in \pi_0^H(\E_{\SU(r)})$ is a unit if the underlying class $x \in \pi_0(\E_\natural)$ is a unit. 
\end{remark}

\medskip
\begin{defn} ($\Z$-graded twisted cohomology of the spectra of broken symmetries) \label{Twist-7}

\noindent
Given a twisting $\theta$ for the family $\E$, we define the twisted cohomology of the spectra $\BC(w_I)$ by 

\[ \sideset{^\theta}{_{\SU(r)}^k}\E(\BC(w_I)) := \pi_0 \Map^{\LSU(r)} (\BC(w_I)^{\Ev}, \Sigma^k \E_{\SU(r)}), \]

\noindent
where the mapping space $\Map^{\LSU(r)} (\BC(w_I)^{\Ev}, \Sigma^k \E_{\SU(r)})$ denotes the space of $\LSU(r)$-equivariant maps from $\BC(w_I)^{\Ev}$ to $\Sigma^k \E_{\SU(r)}$. Note that if $I$ is a positive indexing sequence, so that $\BC(w_I)$ is a $\SU(r)$-space, then $\sideset{^\theta}{_{\SU(r)}^\ast}\E(\BC(w_I))$ is a right module over the untwisted cohomology $\E_{\SU(r)}^\ast(\BC(w_I))$. 
\end{defn}

\medskip
\noindent
As a consequence of definition \ref{DomK3-5} we obtain

\medskip
\begin{thm} \label{Twist-7b}
Assume $\{  \E_{\SU(r)}, \theta_r, r \geq 1 \}$ is an INS-type twisted equivariant cohomology theory (see \cite{Ki1} definition 8.8 for INS-type theories). Then given a link $L$ described as a closure of a braid word $w_I$ on $r$-strands, the filtration on $\sBC(L)^{\Ev}$ gives rise to a cohomologically graded spectral sequence converging to the $\theta$-twisted $\E$-cohomology $\sideset{^\theta}{_{\SU(r)}^\ast}\E(\sBCU(L))$ and with $E_1$-term  
\[ E_1^{t,s} = \bigoplus_{J \in \I^t/\I^{t-1}}    \sideset{^\theta}{_{\SU(r)}^s}\E(\BC(w_J))  \, \, \Rightarrow \, \, \sideset{^\theta}{_{\SU(r)}^{s+t+l(w_I)}}\E(\sBCU(L)). \]
The differential $d_1$ is the canonical simplicial differential induced by the functor in (\cite{Ki1} definition 2.6). Furthermore, the terms $E_q(L,\theta)$ are invariants of the link $L$ for all $q \geq 2$, upto an indeterminacy given by an overall shift in bi-degree. 
\end{thm}

\bigskip
\noindent
Verifying if a twisted theory $\{  \E_{\SU(r)}, \theta_r, r \geq 1 \}$ is an INS-type theory is often a nontrivial problem. In subsequent sections of this article, we will describe several examples of twisted theories and consider the question of showing that they are INS-type theories. 

\section{The theory $\Hh$ and its twist $\theta_\wp$} \label{HT}

\medskip
\noindent
Let us get to the main example of twisted spectra that interests us. Let $\Ho$ denote the (non-equivariant) Eilenberg-MacLane spectrum, and let $\BSU$ denote the infinite complex Grassmannian classifying vector bundles of virtual dimension zero, seen as an infinite loop space under the direct sum of virtual vector bundles. The infinite loop structure of $\BSU$ induces the structure of a commutative ring spectrum on $\Ho \wedge \BSU_+$. The homotopy groups of $\Ho \wedge \BSU_+$ are (by definition) the homology groups of $\BSU$, with an algebra structure induced by the $h$-space structure on $\BSU$. This algebra is a graded polynomial algebra on a collection of generators with one generator in each even degree. One may choose the $\Z$-module spanned by these generators to be the image of 
\[ \iota_\ast : \Ho_\ast(\BSU(1),\Z) \longrightarrow \Ho_\ast(\BSU, \Z), \]
where $\iota : \BSU(1) \longrightarrow \BSU$ is the canonical inclusion. In particular, we have 
\[ \pi_\ast (\Ho \wedge \BSU_+) = \Ho_\ast(\BSU) = \Z[b_1, b_2, \ldots ], \quad |b_i| = 2i, \]
where $b_i$ is the image of the fundamental generator of $\Ho_{2i} (\BSU(1))$ under $\iota_\ast$. We now define an equivariant theory $\Hh$ for which we subsequently construct a natural twist. 

\medskip
\begin{defn} (Borel-equivariant singular cohomology $\Hh$ with parameters) \label{Twist-7c}

\noindent
Define an equivariant cohomology theory $\Hh_{\SU(r)}$ by first constructing the naive $\SU(r)$-equivariant theory given by the Borel completion $\Map(\ESU(r), \Ho \wedge \BSU_+)$, and then inducing up from the trivial to the complete universe.
\end{defn}

\begin{defn} (A twist $\theta_\wp$ for $\Hh$) \label{Twist-6}

\noindent
One checks that $\Hh$ satisfies the conditions of remark \ref{units}. Hence we may classify a twist by an equivariant map from $\SU(r)$ to $\BGL(\Hh_{\SU(r)})$. The group structure of $\BSU$ canonically extends to an equivariant map $\wp$ of group-like monoids, which we may classify
\[ \wp : \Map(\ESU(r), \BSU) \longrightarrow \GL(\Hh_{\SU(r)}), \quad \quad B\wp : \Map(\ESU(r), \SSU) \longrightarrow \BGL(\Hh_{\SU(r)}).  \]
Now let $\SU(r)$ be seen as a space with an $\SU(r)$-action by conjugation. Consider the composite map
\[ \lambda_r : \ESU(r) \times_{\SU(r)} \SU(r) \longrightarrow \ESU \times_{\SU} \SU = \LBSU \llra{\pi_{\SU}} \SU \llra{\pi_{\SSU}} \SSU, \]
where the map $\pi_{\SU} : \LBSU \longrightarrow \SU$ is the projection under the h-space splitting $\LBSU = \SU \times \BSU$. Similarly, $\pi_{\SSU} : \SU \longrightarrow \SSU$ is the projection under the h-space splitting $\SU = \SSU \times \SU(1)$. Let us denote by $\check{\lambda}_r : \SU(r) \longrightarrow \Map(\ESU(r), \SSU))$ the equivariant map given by the adjoint of the map $\lambda_r$ given above. Our twisting class $\theta := \{ \theta_r \}$ is defined as the composite
\[ \theta_r = B\wp \circ \check{\lambda}_r : \SU(r) \longrightarrow \Map(\ESU(r), \SSU)) \longrightarrow \BGL(\Hh_{\SU(r)}).  \]
Since the maps $\pi_{\SU}$ and $\pi_{\SSU}$ are maps of h-spaces, and since this h-space structure is compatible with block-diagonal sums $i_{r,s} : \SU(r) \times \SU(s) \longrightarrow \SU(t)$, it follows that the family $\theta_\wp := \{ \theta_r \}$ serves as a twisting for $\Hh$ as required by definition \ref{Twist-6a} in light of remark \ref{units}. 
\end{defn}

\smallskip
\begin{claim} \label{Twist-6b}
Consider the map induced by $\theta_1$ on the level of fundamental groups 
\[ \pi_1(\theta_1) : \Z = \pi_1 (\SU(1)) \longrightarrow \pi_1(\BGL(\Hh_{\SU(1)})^{\SU(1)}) = \pi_0(\GL(\Hh_{\SU(1)})^{\SU(1)}) \subset [ \BSU(1), \Ho \wedge \BSU_+]. \]
Then the image of $1 \in \Z$ under $\pi_1(\theta_1)$ is the canonical map $\iota : \BSU(1) \longrightarrow \BSU \longrightarrow \Ho \wedge \BSU_+$. Alternatively, interpreting $[ \BSU(1), \Ho \wedge \BSU_+]$ as $(\Ho \wedge \BSU_+)^0(\BSU(1)) = (\Ho \wedge \BSU_+)^0\llbracket x \rrbracket$, the image of $1 \in \Z$ is given by the invertible power series $1 + \wp(x)$, where 
\[ 1 + \wp(x) := \sum_{i \geq 0} b_i x^i, \quad b_0 = 1, \quad \mbox{so that} \quad \wp(x) = \sum_{i \geq 1} b_i x^i.\]
\end{claim}
\begin{proof}
Consider the map $\lambda_r$ for $r=1$
\[ \lambda_1 : \BSU(1) \times \SU(1) \longrightarrow \ESU \times_{\SU} \SU \llra{\pi} \SSU, \]
where $\pi = \pi_{\SSU} \pi_{\SU}$ in terms of the projections defined in definition \ref{Twist-6}. One may describe $\pi$ explicitly as follows. First, we identify $\ESU \times_{\SU} \SU$ as the direct limit of the spaces of the form $\ESU(r)\times_{\SU(r)} \SU(r)$. Consider an equivalence class $[(F,A)] \in \ESU(r)\times_{\SU(r)} \SU(r)$, where $F$ is an $r$-frame in $\C^{\infty}$, and $A$ is a element in $\SU(r)$. Given the pair $(F,A)$, one may construct a unitary matrix in $\SU$ which is defined as the matrix $A$ in the frame $F$, and as the identity matrix in the complement of the subspace spanned by $F$. Taking the limit as $r$ grows, we obtain a map 
\[ \tilde{\pi} : \ESU \times_{\SU} \SU \longrightarrow \SU, \]
which is easily seen to be a map of $h$-spaces, with the $h$-space structure being induced by the block-diagonal sum of matrices. The homomorphism $\pi$ defined above is obtained by projection to $\SSU$. The map $\lambda_1$ given by restricting $\pi$ to $\BSU(1) \times \SU(1)$ is easily seen to factor through a map (which we also denote by the same name)
\[ \lambda_1 : \Sigma \BSU(1) \longrightarrow \SSU. \]
The image of $1 \in \Z$ under $\pi_1(\theta_1)$ is precisely the map induced by the adjoint of $\lambda_1$. On the other hand, the geometric description of $\lambda_1$ given above is well known (in the context of Bott periodicity) to be the map whose adjoint is $\iota$ (see \cite{MT} for instance). The proof of the claim follows. 
\end{proof}

\medskip
\noindent
An easy computation that one can make at this point is the twisted cohomology of the spectra of broken symmetries in the abelian case of $\SU(1)$. Since $\BC(1) = \SU(1)$, we see that $\BC(1)^{\Ev} = \R$, and that the twisted cohomology is a quotient of a graded power series with coefficients in $\Z[b_1, b_2, \ldots]$ which is only non-trivial in odd degrees

\[ \sideset{^{\theta}}{_{\SU(1)}^{2\ast + 1}}\Hh(\SU(1)) = \frac{\Z[b_1, b_2, \ldots]\llbracket x\rrbracket}{\langle \wp(x) \rangle}, \quad \quad \quad \quad \sideset{^{\theta}}{_{\SU(1)}^{2\ast}}\Hh(\SU(1)) = 0 . \]
To see this, we may express $\R$ explicitly as a $\LSU(1)$-CW complex in terms of a homotopy pushout
\[
\xymatrix{
\LSU(1)_+ \wedge_{\SU(1)} (S^0 \vee S^0) \ar[d]^{id \vee id} \ar[r]^{\quad \sigma \vee id} & \LSU(1)_+ \wedge_{\SU(1)} S^0  \ar[d] \\
\LSU(1)_+ \wedge_{\SU(1)} S^0     \ar[r] & \R. 
}
\]
where the left vertical map is induced by the standard map $id \vee id : S^0 \vee S^0 \rightarrow S^0$, and the top horizontal map is the standard map twisted by the right-action of $\sigma$ on one of the factors $ \LSU(1)_+ \wedge_{\SU(1)} S^0$. Applying twisted cohomology and identifying $\sigma$ with $1 + \wp(x)$, we get a Mayer-Vietoris sequence which reduces to the computation given above. 

\medskip
\noindent
The extension to $\SU(1)^{\times r}$ is straightforward

\[ \sideset{^{\theta}}{_{\SU(1)^{\times r}}^{2\ast+r}}\Hh(\SU(1)^{\times r}) = \frac{\Z[b_1, b_2, \ldots]\llbracket x_1, \ldots, x_r \rrbracket}{\langle \wp(x_1), \ldots, \wp(x_r) \rangle}, \quad \quad \sideset{^{\theta}}{_{\SU(1)^{\times r}}^{2\ast+r-1}}\Hh(\SU(1)^{\times r}) = 0. \]

\bigskip
\noindent
We direct the reader to theorem \ref{Twist-8} that generalizes the above computation to spaces of broken symmetries for redundancy free positive sequences $I$. For arbitrary positive sequences, the twisted cohomology $\sideset{^{\theta}}{_{\SU(r)}^\ast}\Hh(\BC(w_I))$ is described in \cite{Go}. 

\medskip
\noindent
The main tool that comes in handy is a spectral sequence induced by a filtration of $\Hh$ by ideals which we now describe. Recall from \cite{HY} that the spectrum $\Sigma^\infty \BSU_+$ has a $\mathbb{A}_\infty$-multiplicative stable splitting, with associated graded object given by $\bigvee_k \MU(k)$. This splitting induces a $\mathbb{E}_2$-monoidal splitting of $\Ho \wedge \BSU_+$ realizing the splitting of $\Z[b_1, b_2, \ldots]$ into homogeneous summands in the generators $b_i$. 

\medskip
\noindent
Consider the (split) filtration of $\BSU_+$ by ideals
\[ F_k (\BSU_+) = \bigvee_{k \leq m} \MU(m), \quad \quad \mbox{where} \quad \MU(0) = S^0.\]
Since the filtration is by ideals, each associated quotient $\MU(k)$ is a module over the ring spectrum $\Sigma^{\infty} \BSU_+$. We make a straightforward observation 

\medskip
\begin{claim} \label{augmentation}
The induced $\Sigma^{\infty} \BSU_+$-module structure on the associated quotient $\MU(k)$ factors through the augmentation map $\Sigma^\infty \BSU_+ \longrightarrow S^0$. 
\end{claim}
\begin{proof}
In terms of the filtration, we notice that $\Sigma^{\infty} \BSU_+ = F_0$. Since the splitting of the filtration is multiplicative, the $F_0$-module structure on $\MU(k) = F_k / F_{k+1}$ factors through the ring spectrum $F_0 / F_1$. The map $F_0 \longrightarrow F_0 / F_1$ is nothing other than the augmentation. 
\end{proof}

\begin{corr} \label{augmentation2}
Let $\Map(\ESU(r), \Ho \wedge \MU(k))$ denote the factor of $\Hh_{\SU(r)}$ induced by the splitting of $\Sigma^\infty \BSU_+$. Then the induced action of $\Map(\ESU(r), \BSU)$ on this summand is trivial. 
\end{corr}

\medskip
\noindent
The filtration $F_k$ of $\Sigma^\infty \BSU_+$ gives rise to an induced $\LSU(r)$-equivariant filtration of $\Hh_{\SU(r)}$ given by $F_k \Hh_{\SU(r)} := \Map(\ESU(r), \Ho \wedge F_k (\BSU_+))$ fitting into a tower of $\LSU(r)$-spectra
\[
\Hh_{\SU(r)} \longrightarrow \cdots \longrightarrow \Hh_{\SU(r)} /F_k \Hh_{\SU(r)} \longrightarrow \cdots \longrightarrow \Hh_{\SU(r)} /F_1 \Hh_{\SU(r)}  = \Ho_{\SU(r)} \longrightarrow  \ast
\]
so that homogeneous layers are of the form $\Map(\ESU(r), \Ho \wedge \MU(k))$. Given a positive indexing sequence $I$, we may take $\LSU(r)$-equivariant maps from $\BC(w_I)^{\Ev}$ into this tower to get a spectral sequence of $\Hh^\ast_{\SU(r)} (\BC(w_I))$-modules. Now using corollary \ref{augmentation2} we notice that $(\Ho \wedge \MU(k))_{\SU(r)} := \Map(\ESU(r), \Ho \wedge \MU(k))$ has a trivial $\OSU(r)$-action. Hence, all $\OSU(r)$-equivariant maps from $\BC(w_I)^{\Ev}$ into any homogeneous layer factor through $\SU(r)$-equivariant maps from $\BC(w_I)$. 

\medskip
\noindent
The above observations lead to a useful theorem

\medskip
\begin{thm} \label{UT2T}
There is a cohomologically graded spectral sequence of $\Hh^\ast_{\SU(r)} (\BC(w_I))$-modules
\[ E_1^{k,j} =   \Hh^{k,j}_{\SU(r)} (\BC(w_I))  \, \, \Rightarrow \, \, \sideset{^{\theta}}{_{\SU(r)}^{k+j}}\Hh(\BC(w_I)),  \]
where $\Hh^{k,j}_{\SU(r)} (\BC(w_I))$ denotes $(\Ho \wedge \MU(k))_{\SU(r)}^{k+j}(\BC(w_I))$. Furthermore, the $\Hh^\ast_{\SU(r)} (\BC(w_I))$-module structure of the spectral sequence shows that the differential $d_1$ is uniquely determined by its value on the class $ 1 \in \Hh^{0,0}_{\SU(r)} (\BC(w_I)) = \Ho^0_{\SU(r)}(\BC(w_I))$. 
\end{thm}

\medskip
\begin{remark} \label{UT2Tb}
Note that by theorem \ref{DomK2-11a}, the groups $\Hh^{k, \ast}_{\SU(r)} (\BC(w_I))$ are the Hochschild homology groups of the Soergel bimodule $\Ho^\ast_{\SU(r)}(\BSa(w_I))$ after one extends the coefficients by homogeneous terms of degree $k$ in the variables $b_i$. 
\end{remark}

\noindent
Our next order of business is to describe the differential $d_1$ in a universal example. Recall that for a positive indexing sequence $I$, the space $\BC(w_I)$ admits a $\SU(r)$-equivariant map
\[ \rho_I : \BC(w_I) \longrightarrow \SU(r). \]
From the construction of $\rho_I$, we see that it factors through the map 
\[ \hat{\rho}_I : \BC(w_I) \longrightarrow \SU(r) \times_T \SU(r), \quad \quad [(g, g_1, \ldots, g_k)] \longmapsto [g, g_1g_2 \ldots g_k], \]
where the $T$-action on the second factor of $\SU(r)$ is by conjugation. Moreover, the maps $\hat{\rho}_I$ are compatible under inclusions of subsets $J \subseteq I$. We therefore define 

\medskip
\begin{defn} (A universal $\LSU(r)$-lift, and a universal spectral sequence) \label{Univlift}

\noindent
Define the $\SU(r)$-space $\UC(r)$ to be $\SU(r) \times_T \SU(r)$ with $T$ acting on the right factor of $\SU(r)$ by conjugation. Note that $\UC(r)$ supports a $\SU(r)$-equivariant map $\UC(r) \longrightarrow \SU(r)$ with $\SU(r)$ acting on itself by conjugation. This endows $\UC(r)$ with a $\LSU(r)$-lift $\UC(r)^{\Ev}$ supporting compatible maps from the spaces $\BC(w_I)^{\Ev}$. Moreover, one has a (universal) spectral sequence
of $\Hh^\ast_{\SU(r)}(\UC(r))$-modules
\[ E_1^{k,j} =   \Hh^{k,j}_{\SU(r)} (\UC(r))  \, \, \Rightarrow \, \, \sideset{^{\theta}}{_{\SU(r)}^{k+j}}\Hh(\UC(r)). \]
\end{defn}

\medskip
\noindent
We now proceed to describe the differential $d_1(1)$ in the universal spectral sequence. 

\medskip
\begin{thm} \label{Univdiff}
Let $B_i \in \Ho^{2i+1}(\SSU)$ denote the canonical primitive generator. Let $g$ denote the composite map 
\[  g : \ET \times_{T} \SU(r) \longrightarrow \ESU(r) \times_{\SU(r)} \SU(r) \llra{\lambda_r} \SSU, \]
where the first map above is the extension of the $T$-action to $\SU(r)$, and the second map is the one defined in \ref{Twist-6}. By identifying $\Ho^\ast_{\SU(r)}(\UC(r))$ with $\Ho^\ast_T(\SU(r))$, we obtain classes $g^\ast B_i \in \Ho^{2i+1}_{\SU(r)}(\UC(r))$ which we also denote by $B_i$. Then the differential $d_1$ in the universal spectral sequence in \ref{Univlift} is determined by 
\[ d_1(1) = \sum_{i \geq 1} b_i B_i \, \in \Hh_{\SU(r)}^{1,0}(\UC(r)). \]
Furthermore, one may restrict $d_1(1)$ to $\Hh_{\SU(r)}^{1,0}(\SU(r) \times_T T)$, along an injection $f^\ast$ induced by 
\[ f: \SU(r) \times_T T \longrightarrow \SU(r) \times_T \SU(r), \quad \quad \mbox{with} \quad \quad f^\ast d_1(1) = \sum_{j=1}^r \wp(x_j) \, \hat{x}_j , \]
where $\hat{x}_j$ and $x_j$ are the standard generators of $\Ho^1(T)$ and $\Ho_T^2$ respectively (see Section \ref{DomK2}). 
\end{thm}
\begin{proof}
Since $f$ is the inclusion of the $T$-fixed points, and $\Ho_T^\ast(\SU(r))$ is free over $\Ho_T^\ast$, it follows from standard facts that $f^\ast$ is injective. Now $f$ induces a map of spectral sequences, and so $f^\ast d_1(1)$ is the first differential for the spectral sequence corresponding to $\SU(r) \times_T T$. The formula for $f^\ast d_1(1)$ follows from claim \ref{Twist-6b}, and the additive defining property of the twist. Now consider the following factorization
\[
\xymatrix{
\ET \times_T T \ar[d]^{\lambda_1^r} \ar[r]^{f} & \ET \times_{T} \SU(r)   \ar[d] \ar[dl]_{g \quad} \\
\SSU & \ESU(r) \times_{\SU(r)} \SU(r)  \ar[l]_{\lambda_r \quad \quad}.
}
\]
The map $\lambda_1$ is well known in the context of Bott periodicity and has the property that $\lambda_1^\ast B_i = x^i \, \hat{x}$, where we express $\Ho^\ast_{\SU(1)}(\SU(1))$ as $\Lambda(\hat{x}) \otimes \Z[x]$ in terms of standard generators $\hat{x}$ and $x$ in degrees $1$ and $2$ respectively. Since $B_i$ is primitive, we see that 
\[  f^\ast B_i = {(\lambda_1^r)}^\ast B_i = \sum_{j=1}^r x_j^i \, \hat{x}_j. \]
Now using the formula $f^\ast d_1(1) = \sum_{j=1}^r \wp(x_j) \, \hat{x}_j = \sum_{j=1}^r \sum_{i = 1}^{\infty} b_i \, x_j^i \, \hat{x}_j$, and interchanging the order of summation, we see that $f^\ast d_1(1)$ is equal to $\sum_{i \geq 1} b_i \, f^\ast B_i$. By the injectivity of $f^\ast$, we have established the result for $\UC(r)$. 
\end{proof}

\section{$sl(n)$-link homology and related spectral sequences} \label{RSS}

\bigskip
\noindent
In the last section we constructed a twist $\theta_\wp$ of the equivariant theory $\Hh$. As mentioned earlier, it is a nontrivial task to show that $\{ \Hh, \theta_\wp \}$ is an INS-type theory. This will indeed be verified in the forthcoming article \cite{Go}. Until then however, it is not apriori obvious that we are allowed to invoke theorem \ref{Twist-7b} to construct link invariants as the pages of the spectral sequence induced by applying the twisted theory $\{ \Hh, \theta_\wp \}$ to the filtration of $\sBC(L)^{\Ev}$. However, in this section, we get around this problem by constructing link invariants using the spectral sequence induced by a hybrid (or total) filtration that incorporates two filtrations; the filtration $\sBC(L)^{\Ev}$ and a filtration of $\Hh$ we describe below. 

\medskip
\noindent
Recall that the tower of $\LSU(r)$-spectra 

\[  \Hh_{\SU(r)} \longrightarrow \cdots \longrightarrow \Hh_{\SU(r)} /F_k \Hh_{\SU(r)} \longrightarrow \cdots \longrightarrow \Hh_{\SU(r)} /F_1 \Hh_{\SU(r)}  = \Ho_{\SU(r)} \longrightarrow \ast. \]

\noindent
Recall also from \ref{DomK3-5} that if $w_I$ is a braid word in the braid group in $r$-strands, then $\sBC(w_I)^{\Ev}$ has its own filtration by $\LSU(r)$-spectra $F_j \sBC(w_I)^{\Ev}$. Hence, one obtains a {\em double tower} of mapping spectra
\[
F^{k,j} \Hh(\sBC(w_I)) := \Map^{\LSU(r)} (F_j \BC(w_I)^{\Ev}, \Hh_{\SU(r)} /F_k \Hh_{\SU(r)}).
\]

\noindent
Define $\{ F^n \Hh(\sBC(w_I)) \}$ as the tower with limit $\Map^{\LSU(r)} (\sBC_\infty(w_I)^{\Ev}, \Hh_{\SU(r)})$ given by the totalization (see \cite{HY} Lemma 5.9) of the double tower $F^{j,k} \Hh(\sBC(w_I))$ above. 
Recalling the terminology of definition \ref{DomK3-5}, we see that the layers of the total tower are of the form
\[ \bigvee_{k+j=n, \, J \in \I^j/\I^{j-1}} \Map^{\SU(r)} (\Sigma^j \BC(w_J)_+, \Map(\ESU(r), \Ho \wedge \MU(k)). \]
Applying homotopy to the total tower $F^n \Hh(\sBC(w_I))$ gives rise to a spectral sequence. After normalizing $\sBC(w_I)$ where $L$ is a closure of $w_I$, (see \ref{DomK3-12}), we get

\medskip
\begin{thm} \label{Rasmussen1}
Let $L$ denote a link given by the closure of a braid word $w_I$ in $r$-strands. Then there is a cohomologically graded spectral sequence induced by the tower $F^n \Hh(\sBC(L))$ with $E_1$-term  
\[ E_1^{t,s}(L,\theta_\wp) = \bigoplus_{r+k=t, \,  J \in \I^r/\I^{r-1}} \Hh_{\SU(r)}^{k,s}(\BC(w_J))  \, \, \Rightarrow \, \, \sideset{^{\theta}}{_{\SU(r)}^{s+t+l(w_I)}}\Hh(\sBC_{\infty}(L)). \]
The first differential $d_1$ is the bicomplex differential $d_1 = d_{\BC} + d_{\Hh}$, where $d_{\BC}$ denotes the first differential of theorem \ref{DomK2-11c}, and $d_{\Hh}$ denotes the first differential in theorem \ref{UT2T}. Furthermore, all terms of the spectral sequence $E_q(L, \theta_\wp)$ are invariants of the link $L$ for $q \geq 2$, with the value of the unknot, for all $q \geq 2$, being
\[ \frac{\Z[b_1, b_2, \ldots]\llbracket x \rrbracket}{\langle \wp(x) \rangle}. \]
\end{thm}

\begin{proof}
The construction of the spectral sequence via the totalization of a double tower shows that the first differential is a bicomplex differential as described. So the only thing we need to show is that the spectral sequence gives rise to link invariants starting with the second page. This can easily be seen by filtering the bicomplex so as to obtain a spectral sequence (this is a version of the Rasmussen spectral sequence; see \cite{Go} for details) that converges to the page $E_2(L, \theta_\wp)$
\[ \Ho (\Ho(\Hh_{\SU(r)}^{k,s} (\BC(w_J), d_{\BC}), d_{\Hh})  \, \, \Rightarrow \, \, E_2(L, \theta_\wp).\]

\noindent
By remark \ref{UT2Tb} and theorem \ref{DomK2-11c}, the inner homology term is a link invariant. It follows from the naturality of the construction of broken symmetries, that the algebraic spectral sequence converges to a link invariant. In particular, $E_2(L, \theta_\wp)$ is a link invariant. Again by naturality, all subsequent terms are link invariants as well. 
\end{proof}

\noindent
In light of theorem \ref{Rasmussen1}, we propose the following conjecture (see \cite{Go} for recent progress) 
\medskip
\begin{conj} \label{Mainconj}
The $E_2(L, \theta_\wp)$ page in the spectral sequence of theorem \ref{Rasmussen1} is a universal $sl(n)$-link homology for any $n \geq 1$. More precisely, setting $b_n = 1$ and $b_i = 0$ for $i \neq n$ (i.e. the specialization to $\wp(x) = x^n$) in the $E_1(L, \theta_\wp)$-term, gives rise to homology groups isomorphic to $sl(n)$-link homology. Furthermore, the bicomplex spectral sequence starting with Triply-graded link homology (with parameters $b_i$) and converging to $E_2(L,\theta_\wp)$ is equivalent to a universal Rasmussen-type spectral sequence (see Theorem 2 in \cite{J2}). 
\end{conj}

\begin{remark}
$E_2(L,\theta_\wp) \Rightarrow \sideset{^{\theta}}{_{\SU(r)}^{s+t+l(w_I)}}\Hh(\sBC_{\infty}(L))$ is a twisted version of $E_k(-1)$ in \cite{J2}.
\end{remark}

\smallskip
\begin{remark} \label{Galois5}
The Galois symmetry $\sigma$ (see remark \ref{Galois4} and example \ref{Galois7} below) gives rise to a symmetry of all pages $E_q(L, \theta_\wp)$. Its action on the value of the unknot is given by the following 
\[ \sigma(x) = -x, \quad \sigma(b_i) = (-1)^{i+1} b_i. \]
\end{remark}

\begin{remark} \label{Rasmussen2}
Notice that one may have run the algebraic spectral sequence in the proof of theorem \ref{Rasmussen1} in the other direction to obtain a spectral sequence of the form 
\[ \Ho (\Ho(\Hh_{\SU(r)}^{k,s} (\BC(w_J), d_{\Hh}), d_{\BC})  \, \, \Rightarrow \, \, E_2(L, \theta_\wp).\]
By theorem \ref{UT2T}, the inner term in the above double homology computes the twisted cohomology of the spectra $\BC(w_J)$. It will be shown in \cite{Go} that the above spectral sequence collapses. Moreover, $E_2(L, \theta_\wp)$ can be identified with the $E_2$ term of the spectral sequence described in Theorem \ref{Twist-7b}. 
\end{remark}

\section{Symmetries and Cohomology operations} \label{Cohops}

\medskip
\noindent
Let us now inquire into the  framework that describes the structure of symmetries and cohomology operations on the twisted theories. In principle, such symmetries and cohomology operations may ``act'' on the actual twist $\theta$, and so one must be careful in making sure that these operations are well defined. Let us take the example of the Galois symmetry $\sigma$ as defined in remark \ref{Galois4} that acts on $\sBC(L)$. This symmetry acts by pointwise complex conjugation on $\LSU(r)$. Using definition \ref{Twist-6a} we may apply $\sigma$ to the twist $\theta$ for a twisted cohomology theory $\{  \E_{\SU(r)}, \theta_r, r \geq 1 \}$. Since $\sigma$ may not preserve $\theta$, we see that $\sigma$ does not give rise to a symmetry of the twisted theory. Fortunately, it is often the case that $\E_{\SU(r)}$ admits a Galois symmetry $\sigma_E$ of its own, so that $\sigma_E \theta \sigma = \theta$. In such situations, one indeed obtains a Galois symmetry on $\{  \E_{\SU(r)}, \theta_r \}$. We given examples below and in section \ref{Appendix2}. 

\medskip
\begin{example} \label{Galois7} (The Galois symmetry on $\Hh$)

\medskip
\noindent
Consider the theory $\Hh_{\SU(r)} = \Map(\ESU(r), \Ho \wedge \BSU_+)$. Consider the Galois symmetry induced by complex conjugation on the spaces $\ESU(r)$ and $\SSU$. The Galois symmetry on $\SSU$ induces a symmetry on $\BSU = \OSSU$\footnote{It is important to note that the Galois symmetry on $\BSU$ when seen as $\OSSU$ is not the usual Galois symmetry on $\BSU$ when one regards it as an infinite Grassmannian.}. These together give rise to a Galois symmetry $\sigma_{\Hh}$ 
\[ \sigma_{\Hh}(z) = (-1)^{k} \quad \mbox{if} \quad z \in \Ho_{\SU(r)}^{2k}(pt), \quad \sigma_{\Hh}(b_i) = (-1)^{i+1} b_i. \]
Now one checks that $\sigma_{\Hh} \theta_{\wp} \, \sigma = \theta_{\wp}$, where $\sigma$ is the Galois symmetry of remark \ref{Galois4}, so that conjugation by the pair of Galois symmetries is seen to preserve the twist $\theta_{\wp}$ as defined in \ref{Twist-6}. This gives rise to a well-defined (conjugate) action of the Galois symmetry on the link invariants $E_q(L, \theta_{\wp})$. 
\end{example}

\medskip
\noindent
Let us now discuss cohomology operations. As with symmetries, these operations may not preserve the twist. In such situations one may work with universal theories where this action is incorporated into the coefficients. Let us describe two important examples. 

\medskip
\begin{example} (The Landweber-Novikov algebra) \label{Witt}

\medskip
\noindent
Recall that the equivariant theory $\Hh$ was defined as the Borel equivariant completion of $\Ho \wedge \BSU_+$. By Thom isomorphism, it is easy to see that $\Ho \wedge \BSU_+ \iso \Ho \wedge \MU$, where $\MU$ denotes the (non-equivariant) complex cobordism spectrum. The theory $\Ho \wedge \MU$ is universal in the sense that its homotopy $\Z[b_1, b_2, \ldots]$ represents (as functors from rings to sets) the data given by a formal group law $F$, endowed with an isomorphism $e_F(x)$ from the additive group law to $F$ \cite{A}. In other words, $e_F(x)$ is the ``universal exponential"
\[ e_F(x) := \sum_{i \geq 0} b_i x^{i+1}, \quad b_0 = 1, \quad \mbox{note that} \,\,   e_F(x) \,\,  \mbox{can be identified with} \quad x + x \, \wp(x). \]
Consider the map $\eta : \Ho \wedge \MU \longrightarrow \Ho \wedge \MU \wedge \MU$ that includes $\MU$ into the right factor. As before, the theory $\Ho \wedge \MU \wedge \MU$ also has a universal description in that its homotopy $\Z[b_i, s_j, i,j \geq 1]$ represents a pair of formal group laws $(F,G)$, and a pair of isomorphisms 
\[ e_F(x) := \sum_{i \geq 0} b_i x^{i+1}, \quad \quad f(x) := \sum_{i \geq 0} s_i x^{i+1}, \quad b_0 = s_0 = 1\]
with $e_F(x)$ being the universal exponential for $F$, and $f(x)$ being an isomorphism from $F$ to $G$. As such, the map $\eta$ represents the composite isomorphism $f(e_F(x))$ from the additive group law to $G$, i.e. the exponential for $G$. Invoking Thom isomorphism, the map $\eta$ represents a map of multiplicative cohomology theories 
\[ \eta : \Ho \wedge \BSU_+ \longrightarrow \Ho \wedge \BSU_+ \wedge \BSU_+. \] 
Performing the Borel equivariant completion along $\eta$ gives rise to a map of equivariant theories endowed with twisting classes $\theta_\wp$ and $\tilde{\theta}_\wp$ respectively 
\[ \eta : \Hh \longrightarrow \tilde{\Hh}, \quad \eta(\theta_\wp) := \tilde{\theta}_\wp, \]
where $\tilde{\Hh}$ is given by equivariantly completing $\Ho \wedge \BSU_+ \wedge \BSU_+$. 
Using the description in terms of formal group laws gives us an explicit description of the power series $\tilde{\wp}(x)$ for the equivariant theory $\tilde{\Hh}$. By the definition of $\eta$, we have
\[ x + x \, \tilde{\wp}(x) = \eta(x + x \, \wp(x)) = \sum_{i\geq 0} \eta (b_i) x^{i+1} = \sum_{i \geq 0} s_i \, (\sum_{j \geq 0} b_j x^j)^{i+1}, \quad b_0 = s_0 = 1. \]
We claim that the map of twisted theories $\eta$ incorporates the data that amounts to a representation of the so-called Landweber-Novikov algebra \cite{M}. In particular, one obtains a representation of the positive Witt algebra $\mathcal{W}$ (which is a sub algebra of the Landweber-Novikov algebra). To see this, let us consider the finite-type $\Z$-dual $\mathcal{W}$ of $\Z[s_1, s_2, \ldots]$. In other words $\mathcal{W}$ is a free $\Z$-module on the duals of the monomials in $s_i$. Then it follows from the above formulas that the map $\eta$ turns $\mathcal{W}$ into an algebra of operators on $\Hh_{\SU(1)}^\ast$. To get a sense of this algebra, let $\mbox{L}_m$ denote the primitive element in $\mathcal{W}$ given by $\mbox{L}_m(s_m)=1$, and defined to evaluate trivially on all other monomials in $s_j$. By the above formula for $\eta$, we get a formula for the action of the operator $\mbox{L}_m$ on the class $y := x + x \, \wp(x)$ defined as
\[ \mbox{L}_m(y) := \mbox{L}_m(\eta(x + x \, \wp(x))) = (\sum_{i \geq 0} b_i x^{i+1})^{m+1} = y^{m+1} \frac{\partial}{\partial y} (y). \]
The relations between the operators $\mbox{L}_m$ show that the sub algebra of $\mathcal{W}$ generated by the elements $\mbox{L}_m$ is isomorphic to the positive Witt algebra. We expect that this action is related to results in \cite{KR}. 
\end{example}

\medskip
\begin{example} (The even Steenrod Algebra) \label{Steenrod}

\medskip
\noindent
In the previous example, we studied the twisting on the Borel equivariant completion $\Hh$ of  $\Ho \wedge \MU$. Let us now define the equivariant theory $\Hh(\F_p)$ as the Borel equivariant completion of $\Ho(\F_p) \wedge \MU$, where $p$ is a prime, and $\Ho(\F_p)$ denotes the mod $p$ Eilenberg-MacLane spectrum. Since we have a map of multiplicative theories
\[ \zeta : \Ho \wedge \MU \longrightarrow \Ho(\F_p) \wedge \MU, \]
induced by the canonical maps on the respective factors, the map $\zeta$ induces a map of equivariant cohomology theories, each endowed with a twist
\[ \zeta : \Hh \longrightarrow \Hh(\F_p) \quad \quad \zeta(\theta_\wp) := \theta_p \]
\noindent
 Notice that $\pi_\ast (\Ho(\F_p) \wedge \MU) = \F_p[b_1, b_2 \ldots]$ induced by $\zeta$. Now consider the inclusion of $\Ho(\F_p)$ as the {\em left} factor $\eta_{\Ah} : \Ho(\F_p) \wedge \MU \longrightarrow \Ho(\F_p) \wedge \Ho(\F_p) \wedge \MU$ inducing a map 
\[ \eta_{\Ah} : \Hh(\F_p) \longrightarrow \Ah(\F_p), \quad \quad \eta_{\Ah}(\theta_p) := \theta_{\Ah}, \]
where $\Ah(\F_p)$ is the equivariant theory obtained on completing $\Ho(\F_p) \wedge \Ho(\F_p) \wedge \MU$. The homotopy of $\Ho(\F_p) \wedge \Ho(\F_p) \wedge \MU$ is of the form
\[ \pi_\ast (\Ho(\F_p) \wedge \Ho(\F_p) \wedge \MU) = \pi_\ast (\Ho(\F_p) \wedge \Ho(\F_p)) \, [b_1, b_2, \ldots ], \]
where the homotopy groups of the spectrum $\Ho(\F_p) \wedge \Ho(\F_p)$ is none other than the dual mod-$p$ Steenrod algebra. As such $\pi_\ast (\Ho(\F_p) \wedge \Ho(\F_p))$ is an algebra on free variables $\xi_i$ for $i \geq 1$ in homological degree $2(p^i-1)$, as well as a family of free variables in odd degree.
As in the previous example, the subalgebra of $\pi_\ast (\Ho(\F_p) \wedge \Ho(\F_p) \wedge \MU)$ generated by the classes $\xi_i$ and $b_j$ represents the data given by a pair of isomorphisms $a(x)$ and $e_F(x)$ of formal group laws over $\F_p$, where $e_F(x)$ is as before, and $a(x)$ is an automorphism of the additive formal group law \cite{I,W}
\[ a(x) = \sum_{i \geq 0} \xi_i x^{p^i}, \quad \xi_0 = 1. \]
As in the previous example, $\eta_{\Ah}$ represents composition of isomorphisms. We may apply $\eta_{\Ah}$ to $\Hh(\F_p)_{\SU(1)}^\ast$ and describe $\eta_{\Ah}$ applied to the class $x + x \, \wp(x)$ by a formula 
\[ x + x \, \wp_{\Ah}(x) = \eta_{\Ah}(x + x \, \wp(x)) = \sum_{i\geq 0} \eta_{\Ah} (b_i) x^i = \sum_{i \geq 0} b_i \, (\sum_{j \geq 0} \xi_j x^{p^j})^i, \quad \xi_0 = b_0 = 1. \]
The reader may verify that this formula describes a co-action of the dual even Steenrod algebra on $\Hh(\F_p)_{\SU(1)}^\ast$ extending the standard co-action on $\Ho(\F_p)_{\SU(1)}^\ast = \F_p[x]$. 
\end{example}
\section{Appendix: Calculations in Borel equivariant cohomology} \label{Appendix}

\medskip
\noindent
The task we aim to achieve in the Appendix is to study the Borel equivariant cohomology of spaces of the form $\BC(w_I)$ for some positive indexing sequence $I$. The answer will be expressed in terms of the Borel equivariant  cohomology of the Bott-Samelson spaces $\BSa(w_I)$. And so we will begin with the structure of the latter. Consider the Bott-Samelson variety $\BSa(\sigma_i) = \SU(r) \times_T (G_i/T)$, where $\sigma_i \in \Br(r)$ is the standard generator for $1 \leq i < r$. 

\medskip
\begin{thm} \label{DomK3-17}
$\Ho_{\SU(r)}^\ast(\BSa(\sigma_i))$ is a rank two free module over $\Ho_T^\ast$, generated by classes $\{ 1, \delta \}$, where $\delta \in \Ho_{\SU(r)}^0(\BSa(\sigma_i))$ is uniquely defined by the property $\delta^2 + \alpha \delta = 0$, where $\alpha \in \Ho_T^2$ is the character $x_i - x_{i+1}$ corresponding to the root $\alpha$. 
\end{thm}

\begin{proof}
The proof of theorem \ref{DomK3-17} is classical. Consider the two $T$-fixed points of $G_i/T$ given the the cosets $T/T$ and $\sigma_i T/T$. The normal bundle of $\sigma_i T$ in $G_i/T$ is isomorphic to the $T$-representation with linear character given by the dual $\overline{\alpha}$ of $\alpha$. The two fixed points $T/T$ and $\sigma_i T/T$ give rise to two sections $s$ and $s_{\sigma_i}$ respectively of the bundle 
\[ \SU(r) \times_T (G_i/T) \longrightarrow \SU(r)/T. \]
Pinching off the section $s$ gives rise to a cofiber sequence which splits into short exact sequences in equivariant cohomology
\[ 0 \longrightarrow \Ho_T^\ast(\Sigma^{\overline{\alpha}}) \llra{f^\ast} \Ho^\ast_{\SU(r)}(\BSa(\sigma_i)) \llra{s^\ast} \Ho_T^\ast \longrightarrow 0. \]
Let $\lambda \in \Ho^\ast_{\SU(r)}(\BSa(\sigma_i))$ be the class given by the image to the Thom class of the representation $\overline{\alpha}$ under the map $f^\ast$. We see from from the above sequence that $\{1, \lambda \}$ are $\Ho_T^\ast$-module generators of $\Ho_{\SU(r)}^\ast(\BSa(\sigma_i))$. Consider the map induced by the inclusion of fixed points
\[ s \sqcup s_{\sigma_i} : (\SU(r)/T) \sqcup (\SU(r)/T) \longrightarrow \BSa(\sigma_i). \]
The above short exact sequence shows that $\lambda$ is uniquely determined by the fact that it restricts trivially along $s^\ast$ (by construction) and restricts to the element $-\alpha$ along $s_{\sigma_i}^\ast$ since $-\alpha$ is the Euler class of the representation $\overline{\alpha}$. The generator $1 \in \Ho_{\SU(r)}^0(\BSa(\sigma_i))$ clearly restricts to $1 \in \Ho_T^0$ along both fixed points. It now follows that any element in $\Ho_{\SU(r)}^\ast(\BSa(\sigma_i))$ is uniquely determined by its restrictions along these two fixed points. Applying this observation to $\lambda^2$ yields the relation $\lambda^2 + \alpha \lambda = 0$. 
Our generator $\delta$ is simply defined as $\lambda$. It is straightforward to see that $\delta$ is the unique class that satisfies 
\[ \delta^2 + \alpha \,  \delta = 0. \]
\end{proof}

\begin{remark} \label{DomK3-18}
Consider the map
\[ \tau_i : \SU(r) \times_T (G_i/T) \longrightarrow \SU(r)/T, \quad [(g, g_i T)] \longmapsto gg_i T. \]
Then, given $\gamma \in \Ho^\ast_{\SU(r)}(\SU(r)/T)) = \Ho_T^\ast$, we my ask to express the element $\tau_i^\ast(\gamma)$ in terms of our generators $\{1, \delta \}$. This is easily done by using the restrictions along the two fixed points. Expressing $\tau_i^\ast (\gamma)$ as $\tau_i^\ast (\gamma) = a + b \, \delta$, 
we may restrict along $s^\ast$ to deduce that $a = \gamma$. Then restricting along $s_{\sigma_i}^\ast$ says that 
\[ b = \frac{\gamma - \sigma(\gamma)}{\alpha}. \]
\end{remark}

\medskip
\begin{defn} \label{Bott-Samelson} 
Recall the definition \ref{DomK2-10b} of Bott-Samelson varieties and their Schubert classes. Given $I = \{i_1, \ldots, i_k\}$, and any $j \leq k$, let $J_j = \{i_1, \ldots, i_j \}$. Define maps $\pi(j)$ and $\tau(j)$ 
\[ \pi(j) : \BSa(w_{J_j}) \longrightarrow \BSa(\sigma_{i_j}), \quad [(g, \, g_{i_1}, \ldots, g_{i_j})] \longmapsto [(gg_1\cdots g_{i_{j-1}}, \, g_{i_j})] \]
\[   \tau(j) :  \BSa(w_{J_j}) \longrightarrow \SU(r)/T, \quad \, \, \, \,  [(g, \, g_{i_1}, \ldots, g_{i_j})] \longmapsto [gg_1\cdots g_{i_j}]. \quad \quad \quad \]
Notice that one has a pullback diagram with vertical maps being canonical projections
\[
\xymatrix{
\BSa(w_{J_j}) \ar[d] \ar[r]^{\pi(j)} &  \BSa(\sigma_{i_j}) \ar[d] \\
    \BSa(w_{J_{j-1}}) \ar[r]^{\tau(j-1)} & \SU(r)/T. 
}
\]
Since $\BSa(w_I)$ projects canonically onto $\BSa(w_{J_j})$, we will use the same notation to denote the maps $\pi(j)$ and $\tau(j)$ with domain $\BSa(w_I)$. Define $\delta_j \in \Ho_{\SU(r)}^2(\BSa(w_I))$ as the pullback class $\pi(j)^\ast(\delta)$, where $\delta \in \Ho_{\SU(r)}^2(\BSa(\sigma_{i_j}))$ is the (unique) generator of $\Ho_{\SU(r)}^2(\BSa(\sigma_{i_j}))$ as a $\Ho_T^\ast$-module that satisfies $\delta^2 = -\alpha \, \delta$. Here $i_j = s$, and $\alpha := x_s - x_{s+1}$ denotes the character of the standard positive simple root $\alpha_s$ of $\SU(r)$ that corresponds to the subgroup $G_{i_j}$. In particular, we have 
\[ \delta_j^2 + [\alpha]_j \,  \delta_j = 0, \quad \mbox{where we define} \quad [x]_j := \tau(j-1)^\ast(x). \]
\end{defn}

\medskip
\begin{defn} (The redundancy set $\nu(I)$, its complement $\overline{\nu}(I)$, and the blocks of $I$) \label{DomK2-10c}

\noindent
Let $I = \{ i_1, i_2, \ldots, i_k \}$ denote a positive indexing sequence with each $i_j < r$. Let $\nu(I)$ be the (unordered) set of integers $s<r$ such that $s$ occurs somewhere in $I$ and denote by $\nu(s)$ the number of times $s$ occurs in $I$.  We say $I$ is redundancy free if $\nu(s)=1$ for all $s \in \nu(I)$. 

\smallskip
\noindent
Let $W_I \subseteq \Sigma_r$ be the subgroup of the Weyl group of $\SU(r)$ generated by the reflections $\sigma_s$ for $s \in \nu(I)$. $W_I$ is a standard subgroup in $\Sigma_r$ of the form $\Sigma_{r_1} \times \ldots \times \Sigma_{r_m}$, where $m$ is the number $W_I$-orbits in $\{1, \ldots, r \}$, and $r_i$ is their sizes. Note that this block decomposition corresponds to the fact that the groups $G_i$ for $i \in I$ generate the block diagonal subgroup $\SU(r_1) \times \ldots \times \SU(r_m) \subseteq \SU(r)$. 

\smallskip
\noindent
We call the $W_I$-orbits in $\{ 1, \ldots, r \}$ the blocks of $I$, and define $\overline{\nu}(I)$ as the collection of all blocks. Since all non-maximal elements $s$ in a block belong to $\nu(I)$, we see that $\overline{\nu}(I)$ is in bijection with the complement of the set $\nu(I)$ in $\{ 1, \ldots, r \}$, and hence the notation. Given $s \leq r$, we denote $l(s)$ the block in which $s$ belongs. 
\end{defn}

\medskip
\begin{thm} \label{DomK3-20}
Let $I = \{ i_1, i_2, \ldots, i_k \}$ denote a positive indexing sequence with $i_j < r$. Then 
\[ \Ho_{\SU(r)}^\ast(\BSa(w_I)) = \frac{\Ho^\ast_T[\delta_{1}, \delta_{2}, \ldots, \delta_{k}]}{\langle \delta_{j}^2 + [\alpha_s]_j \, \delta_{j}, \, \, \mbox{if} \, \, i_j \in I_s. \rangle }. \]
Moreover, for any weight $\alpha$, the character $\alpha \in \Ho_T^2$ satisfies the following recursion relations
\[  [\alpha]_1 = \alpha \quad \mbox{and} \quad [\alpha]_j = [\alpha]_{j-1} \,  + \, \alpha(h_u) \, \delta_{j-1}, \, \, \, \mbox{where} \, \, i_{j-1} = u, \]
with $\alpha(h_u)$ denoting the value of the weight $\alpha$ evaluated on the coroot $h_u$. Furthermore, the behaviour under inclusions $J \subseteq I$, is given by setting all $\delta_t = 0$ for $i_t \in I/J$. 
\end{thm}
\begin{proof}
The proof of the above theorem is a simple induction argument. Let $I = \{ i_1, \ldots, i_k \}$ as above. Assume that the classes $\{ \delta_{1}, \ldots, \delta_{k} \}$ have been constructed as in definition \ref{Bott-Samelson}. Now consider the indexing subsequence $J_{k-1} =\{ i_1, \ldots, i_{k-1} \}$. One has $G_{i_k}/T$-bundles
\[ \BSa(w_I) \longrightarrow \BSa(w_{J_{k-1}}), \quad [(g, g_{i_1}, \ldots, g_{i_k})] \longmapsto [(g, g_{i_1}, \ldots, g_{i_{k-1}})]. \]
As before, the above fibration supports two sections $s_J$ and $s_{J,\sigma_k}$ induced by the two cosets $\{ T/T, \sigma_{i_k} T/T \} \subset G_{i_k}/T$. Furthermore, using definition \ref{Bott-Samelson} one has a diagram of cofiber sequences induced by the inclusion of the section $s_J$
\[
\xymatrix{
 \BSa(w_{J_{k-1}})  \ar[r]^{s_J} \ar[d]^{\tau(k-1)} & \BSa(w_I) \ar[r] \ar[d]^{\pi(k)} & \Sigma^{\overline{\alpha}_{i_k}}\BSa(w_{J_{k-1}}) \ar[d]^{\tau(k-1)}   \\
 \SU(r)/T \ar[r]^{s} & \BSa(w_{i_k}) \ar[r] & \Sigma^{\overline{\alpha}_{i_k}} \SU(r)/T.
}
\]
By induction, the classes $\{\delta_{1}, \ldots, \delta_{k-1} \}$ restrict to generators of $\Ho_{\SU(r)}^\ast(\BSa(w_{J_{k-1}}))$. So by degree reasons, we see that the above diagram gives rise to a diagram of short exact sequences in equiariant cohomology. In particular, we see that $\Ho_{\SU(r)}^\ast(\BSa(w_I))$ is a free module of rank two on $\Ho_{\SU(r)}^\ast(\BSa(w_{J_{k-1}}))$ generated by the classes $\{1, \lambda_k \}$, where $\lambda_k$ is the image of the Thom class in $\Ho_{\SU(r)}^\ast(\Sigma^{\overline{\alpha}_{i_k}}\BSa(w_{J_{k-1}}))$. By theorem \ref{DomK3-17}, and using naturality, we see that this class is $\delta_{k}$. 
This completes the induction argument. It remains to prove the recursive relation formula. This follow from unraveling the formula in remark \ref{DomK3-18}. 
\end{proof}

\medskip
\begin{thm} \label{DomK3-23}
Let $I = \{ i_1, i_2, \ldots, i_k \}$ denote a positive indexing sequence with each $i_j < r$. Then, for any $s \in \nu(I)$, with $\nu(s) = 1$, there exists a canonical class $\beta_s \in \Ho_{\SU(r)}^3(\BC(w_I))$. Likewise, for any block $l \in \overline{\nu}(I)$, there exists a canonical class $\gamma_l \in \Ho_{\SU(r)}^1(\BC(w_I))$. 

\smallskip
\noindent
Let $\beta_I$ and $\gamma_I$ denote the set of all generators defined above. Then $\Ho_{\SU(r)}^\ast(\BC(w_I))$ is a free module over the algebra $\Ho_T^\ast \otimes \Lambda(\beta_I) \otimes \Lambda(\gamma_I)$. 

\smallskip
\noindent
Moreover, given a subset $J \subset I$ obtained by deleting an index $i_j$ so that $i_j = s$ and $\nu(s) = 1$, the restriction from $\Ho_{\SU(r)}^\ast(\BC(w_I))$ to $\Ho_{\SU(r)}^\ast(\BC(w_{J}))$ extends to an isomorphism 
\[ \varpi_I : \frac{\Ho_{\SU(r)}^\ast(\BC(w_I)) \otimes \Lambda(\gamma \, )}{\langle \, \, \beta_s - [\alpha_s]_j \, \gamma \, \,  \rangle} \llra{\cong} \Ho_{\SU(r)}^\ast(\BC(w_{J})), \quad \quad \gamma \longmapsto \gamma_{l(s)}. \]
In other words, switching from $\Ho_{\SU(r)}^\ast(\BC(w_I))$ to $\Ho_{\SU(r)}^\ast(\BC(w_{J}))$ swaps one exterior generator $\beta_s$ in exchange for an exterior generator $\gamma_{l(s)}$. 
In the above isomorphism, if $k > s$ is any integer in the same $W_I$-block as $s$, with $\nu(k)=1$ and $i_t = k$, then $\beta_k$ restricts to $\beta_k + [\alpha_k]_t \gamma_{l(s)}$. Finally, $\gamma_{l(s)}$ restricts to $\gamma_{l(s)} + \gamma_{l(s+1)}$. Other generators restrict to their namesakes. 
\end{thm}

\begin{proof}
Recall from the proof of theorem \ref{DomK2-11a} that one has a Serre spectral sequence converging to $\Ho_{\SU(r)}^\ast(\BC(w_I))$ with $E_2$-term 
\[ \Lambda(\epsilon_1, \ldots, \epsilon_r) \otimes \Ho_{\SU(r)}^\ast(\BSa(w_I)), \quad d(\epsilon_k) = \pi(I)^\ast(h_k^\ast) - \tau(I)^\ast(h^\ast_k), \]
where $\tau(I)$ and $\pi(I)$ were the maps defined in definition \ref{DomK2-10b}. This complex is the standard Koszul complex that computes the Hochschild homology of the bimodule $\Ho_{\SU(r)}^\ast(\BSa(w_I))$ with coefficients in $\Ho_T^\ast$. Since this spectral sequence collapses at $E_3$, and is free over $\Ho_T^\ast$ (as shown in the proof of theorem \ref{DomK2-11a}), we recover the cohomology ring $\Ho_{\SU(r)}^\ast(\BC(w_I))$ up to choices of lifts. We will show in what follows that one may make canonical choices for the classes $\beta_s$ and $\gamma_l$. 

\medskip
\noindent
Let us replace the generators $\epsilon_k$ by an alternate set of generators we denote by $\hat{\tau}_q$ and $\gamma_l$ for $q \in \nu(I)$ and $l \in \overline{\nu}(I)$, where 
\[ \hat{\tau}_q := \sum_{j \leq q, \, j \in l(q)} \hat{x}_j, \quad \quad \gamma_l := \sum_{j \in l} \hat{x}_j. \]
The classes $\gamma_l$ are easily seen to be parmanent cycles. For any index $s \in \nu(I)$ with $\nu(s) = 1$, let $i_j \in I$ be so that $i_j = s$. Then the recurrence relation given in theorem \ref{DomK3-20} shows that 
\[ d(\hat{\tau}_s) = \delta_j. \]
It follows that the class $\beta_s := \hat{\tau}_s(\delta_j + [\alpha_s]_j)$ is a permanent cycle. This class $\beta_s$ is bidegree $(2,1)$ and the class $\gamma_l$ is in bidegree $(0,1)$. There are no classes in lower filtration in degrees 3 and 1 respectively, therefore $\beta_s$ and $\gamma_l$ lift uniquely to $\Ho_{\SU(r)}^\ast(\BC(w_I))$. This defines the classes $\beta_s$ and $\gamma_l$ that we seek in degrees $3$ and $1$ respectively. 

\medskip
\noindent
Let $\beta_I$ and $\gamma_I$ denote the respective sets of the elements constructed above. One easily notices that the $E_3$-term of the spectral sequence is a free module over $\Ho_T^\ast \otimes \Lambda(\gamma_I)$. It follows that $\Ho_{\SU(r)}^\ast(\BC(w_I))$ is a free module over the algebra $\Ho_T^\ast \otimes \Lambda(\gamma_I)$.

\medskip
\noindent
To prove the rest of the theorem, we argue by induction starting with the maximal subsequence $I_0 \subseteq I$ with the property that $\nu(s) > 1$ for any $s \in \nu(I_0)$. In other words, $\Ho_{\SU(r)}^\ast(\BC(w_{I_0}))$ has no classes of the form $\beta_s$. From the above discussion, the theorem is true for $I_0$. 

\medskip
\noindent
Now assume that $J \subset I$ obtained by deleting an index $i_j$ so that $i_j = s$ and $\nu(s) = 1$. Also assume that the theorem is true for $J$. Removing the index $s$ from $\nu(I)$ splits the block $l(s)$ in $\overline{\nu}(I)$ into the two blocks in $\overline{\nu}(J)$ given by $l(s)$ and $l(s+1)$. Namely, the block $l(s) \in \nu(I)$ decomposes into a union of two $W_{J}$ orbits: $l(s) \cup l(s+1)$. It follows that the class $\gamma_{l(s)}$ restricts to $\gamma_{l(s)} + \gamma_{l(s+1)}$ under the restriction from $\Ho_{\SU(r)}^\ast(\BC(w_I))$ to $\Ho_{\SU(r)}^\ast(\BC(w_{J}))$. Similarly, comparing Serre spectral sequences, we see that the class $\beta_s$ restricts to $[\alpha_s]_j \gamma_{l(s)}$. The restriction for $\beta_k$ for $k>s$ with $\nu(k) = 1$, and sharing a block with $s$ is proved in the same fashion. This proves that the map $\varpi_I$ as stated in the theorem is well defined. 

\medskip
\noindent
It remains to show that $\varpi_I$ is an isomorphism, and that the cohomology $\Ho_{\SU(r)}^\ast(\BC(w_I))$ has an extra free generator $\beta_s$ in exchange for the generator $\gamma_{l(s)}$. Both these facts will follow easily once we show that $\Ho_{\SU(r)}^\ast(\BC(w_I))$ is a free module over the exterior algebra $\Lambda(\beta_s)$ with the space of generators given by $\Ho_{\SU(r)}^\ast(\BC(w_J))/\langle \gamma_{l(s)} \rangle$. 

\medskip
\noindent
Using the first Markov property, we may work with the special case where $J$ is obtained from $I$ by dropping the last index so that $J = \{ i_1, \ldots, i_{k-1} \}$. Let us express the block-diagonal group $G_{i_k} \subseteq \SU(r)$ as the group $T^{r_1} \times \SU(2) \times T^{r-r_1-2}$ in the standard basis of $\C^r$ for some $r_1$. We consider a product decomposition of the diagonal torus $T^r \subset \SU(r)$, which is of the form $T^r = T^{r-1} \times \Delta$, where $\Delta$ denotes the standard diagonal maximal torus in the semi-simple factor $\SSU(2) \subset G_{i_k}$, and $T^{r-1}$ is any rank-$(r-1)$ subtorus of $T^r$ that contains the maximal tori of each semi-simple factor $\LV_{i} \subset G_{i}$ for $i \neq i_k$.

\medskip
\noindent
Having chosen the above decomposition, notice that we may express any block-diagonal group $G_{i}$ for $i \neq i_k$ in the form $G_{0,i} \rtimes \Delta$, where $G_{0,i}$ is the group generated by $\LV_{i}$ and $T^{r-1}$. In particular, we have a decomposition of $\BC_T(w_J)$ as topological $T$-spaces, with $T$-acting by conjugation on both factors
\[ \BC_T(w_J) = G_{i_1} \times_T G_{i_2} \ldots \times_T G_{i_{k-1}} = (G_{0,i_1} \times_{T^{r-1}} G_{0,i_2} \ldots \times_{T^{r-1}} G_{0,i_{k-1}}) \times \Delta := \BC_{T^{r-1}}(w_J) \times \Delta . \]
The above decomposition extends to a decomposition of $\BC_T(w_I)$ as $T$-spaces 
\[ \BC_T(w_I) = (G_{0,i_1} \times_{T^{r-1}} G_{0,i_2} \ldots \times_{T^{r-1}} G_{0,i_{k-1}}) \times \SSU(2) = \BC_{T^{r-1}}(w_J) \times \SSU(2). \]
Next, consider the map of fibrations induced by the projection maps onto the first factor
\[
\xymatrix{
\Delta \ar[d]^f \ar[r] & \SSU(2) \ar[d]^g \\
 \ET\times_T(\BC_T(w_J)) \ar[d]^{p_J} \ar[r] &  \ET \times_T(\BC_T(w_I))  \ar[d]^{p_I} \\
 \ET\times_T(\BC_{T^{r-1}}(w_J))  \ar[r]^{=} & \ET\times_T(\BC_{T^{r-1}}(w_J)) .
}
\]
The Serre spectral sequences for both fibrations collapse because the maps $p_J$ and $p_I$ admit splittings. Furthermore, it is easy to see that $g^\ast$ maps the class $\beta_s$ to a generator of $\Ho^3(\SSU(2))$. Similarly, $f^\ast$ is easily seen to map the class $\gamma_{l(s)}$ to a generator of $\Ho^1(\Delta)$. It follows from the collapse of the Serre spectral sequence for the fibration $p_I$ that $\Ho_{\SU(r)}^\ast(\BC(w_I))$ is a free module over the exterior algebra $\Lambda(\beta_s)$ with the space of generators given by $\Ho_{T}^\ast(\BC_{T^{r-1}}(w_J))$. Now from the collapse of the Serre spectral sequence for $p_J$ we can (abstractly) identify this space of generators with $\Ho_{\SU(r)}^\ast(\BC(w_J))/\langle \gamma_{l(s)} \rangle$. 
\end{proof}

\smallskip
\begin{remark} \label{restriction}
Consider the restriction map $\Ho_{\SU(r)}^\ast(\BC(w_I)) \longrightarrow \Ho_T^\ast(T)$. This map is injective if $\nu(s) = 1$ for all $s \in \nu(I)$, namely if $I$ is redundancy free. The discussion in the proof of theorem \ref{DomK3-23} describes the restriction applied to the classes $\gamma_l$ and $\beta_s$ 
\[ \gamma_l \longmapsto \sum_{i \in l} \hat{x}_i, \quad  \quad \quad \beta_s \longmapsto \alpha_s \, \hat{\tau}_s = \sum_{i \leq s, \, i \in l(s)} \alpha_s \, \hat{x}_i.\]
\end{remark}

\noindent
The twisted cohomology of the spaces $\BC(w_I)$ is subtle. We may describe the twisted cohomology for redundancy free positive sequences $I$. The general case is described in \cite{Go}. Given a simple positive root for $\SU(r)$, of the form $\alpha_i = x_i - x_{i+1}$, we define
\[ \wp(\alpha_i) := \frac{\wp(x_i) - \wp(x_{i+1})}{x_i - x_{i+1}}. \]
\begin{thm} \label{Twist-8}
Let $I = \{ i_1, i_2, \ldots, i_k \}$ be a redundancy free positive indexing sequence with each $i_j < r$. Let $\Hh$ denote the Borel equivariant completion of $\Ho \wedge \BSU_+$. Then we have a canonical isomorphism of $\Hh_{T}^\ast(pt) = \Z[b_1, b_2, \ldots]\llbracket x_1, \ldots, x_r \rrbracket$-modules 
\[  \sideset{^{\theta}}{_{\SU(r)}^{\ast + r + 2|\nu(I)|}}\Hh(\BC(w_{I})) \cong \frac{\Z[b_1, b_2, \ldots]\llbracket x_1, \ldots, x_r \rrbracket}{\langle  \,  \wp(\alpha_s), \, \,  \mbox{if} \, \, s \in \nu(I), \, \,  \wp(x_{m(l)}),   \, \mbox{if} \, \, l \in \overline{\nu}(I) \, \rangle}, \]
where $m(l)$ denotes the maximal element in the block $l \in \overline{\nu}(I)$. For $I$ as above, let $J = I - \{i_t \}$ with $i_t = s$. Then the restriction map $\sideset{^{\theta}}{_{\SU(r)}^{\ast}}\Hh(\BC(w_I)) \longrightarrow \sideset{^{\theta}}{_{\SU(r)}^{\ast}}\Hh(\BC(w_J))$ injects onto the submodule of elements divisible by $\alpha_{s}$. 
\end{thm}

\begin{proof}
We recall from theorem \ref{UT2T} the spectral sequence of modules over $\Hh_{\SU(r)}^\ast(\BC(w_I))$ that converges to $\sideset{^{\theta}}{_{\SU(r)}^\ast}\Hh(\BC(w_I))$. This spectral sequence starts with the untwisted equivariant cohomology theory $\Hh_{\SU(r)}^\ast(\BC(w_I))$, which we completely understand by virtue of theorem \ref{DomK3-23} to be a free module over the exterior algebra $\Lambda(\beta_I) \otimes \Lambda(\gamma_I)$ generated by the space $\Hh_{T}^\ast(pt)$. We remind the reader that to construct this spectral sequence, we filter $\Ho_\ast(\BSU_+)$ by powers of the ideal generated by the elements $b_1, b_2, \ldots$. This filtration is induced by a stably-split topological filtration of $\BSU_+$ by ideals that give rise to a filtration of $\sideset{^{\theta}}{_{\SU(r)}^\ast}\Hh(\BC(w_I))$ whose associated quotient is $\Hh_{\SU(r)}^\ast(\BC(w_I))$, since the twisting is trivial under the associated quotient. 

\smallskip
\noindent
The first differentia in the above spectral sequencel is determined by what it does on the class ``1", which we express in terms of the generators defined in theorem \ref{DomK3-23}
\[ d(1) = \sum_{s \in \nu(I), \, l \in \overline{\nu}(I)} c_s \, \beta_s + d_l \, \gamma_l, \]
for some coefficients $c_s$ and $d_l$. Recall from \ref{DomK3-23} that the restriction of $\Hh_{\SU(r)}^\ast(\BC(w_I))$ to $\Hh_{T}^\ast(T)$ is injective. Hence, we can identify the coefficients $c_s$ and $d_l$ by unraveling the equality in $\Hh_{T}^\ast(T)$
\[ \sum_{s \in \nu(I), \, l \in \overline{\nu}(I)} (c_s \, \alpha_s \, \tau_s + d_l  \, \gamma_l) = \sum_i \wp(x_i) \, \hat{x}_i, \]
where the right hand side is seen to be differential using theorem \ref{Univdiff}. Since the classes $\tau_s$ and $\gamma_l$ are linearly independent over the ring $\Hh_{T}^\ast(pt)$, an elementary calculation using remark \ref{restriction} shows that $d_l = \wp(x_{m(l)})$, where $m(l)$ is the maximal element in the block $l$. Similarly, $c_s = \wp(\alpha_s)$. We may now compute the cohomology under the first differential, and verify that it is a quotient of the ideal in $\Hh_{\SU(r)}^\ast(\BC(w_I))$ generated by the product class $\prod_{s,l} \beta_s \gamma_l$ in degree $r+2|\nu(I)|$. 

\smallskip
\noindent
By parity reasons, the above spectral sequence collapses. The result is easily seen to be a cyclic submodule for $\Hh_{T}^\ast(pt)$ with a shift in degree given by $r+2|\nu(I)|$, and is canonically isomorphic to 
\begin{equation} \label{redfree} \frac{\Z[b_1, b_2, \ldots]\llbracket x_1, \ldots, x_r \rrbracket}{\langle  \,  \wp(\alpha_s), \, \,  \mbox{if} \, \, s \in \nu(I), \, \,  \wp(x_{m(l)}),   \, \mbox{if} \, \, l \in \overline{\nu}(I) \, \rangle}. \end{equation}
Naturality under restriction also follows from the above description. We have therefore proven the theorem. 
\end{proof}

\medskip
\begin{example} \label{bootstrap}
In certain instances, one may explicitly compute the cohomology of spaces $\BC(w_I)$ even when $I$ is not redundancy free. For instance, consider the indexing sequence of the form $I_s = (s, s, \cdots, s)$ with only one index repeated $k$ times. In this special case, one has a pullback diagram
\[
\xymatrix{
\BC(w_{I_s}) \ar[d] \ar[r]^{M} &  \BC(\sigma_s) \ar[d] \\
    \BSa(w_{I_s}) \ar[r]^{M} & \BSa(\sigma_s). 
}
\]
where the vertical maps the the canonical maps from the spaces of broken symmetries to the Bott Samelson spaces and the horizontal maps $M$ are induced by the $k$-fold group multiplication in $G_s$. It is straightforward to see that the maps $M$ are split, and that the map $M : \BSa(w_{I_s}) \longrightarrow \BSa(\sigma_s)$ has the property
\[ M^\ast (\delta) = \hat{\delta} := \delta_1 + \delta_2 + \cdots + \delta_k. \]
It follows from the Eilenberg-Moore spectral sequence that 
\[ \Ho_{\SU(r)}^\ast(\BC(w_{I_s})) = \Ho_{\SU(r)}^\ast( \BC(\sigma_s)) \otimes_{\Ho^\ast_{\SU(r)}(\BSa(\sigma_s))} \Ho_{\SU(r)}^\ast( \BSa(w_{I_s}) ). \]
A similar argument shows that in the twisted case
\[ \sideset{^{\theta}}{_{\SU(r)}^\ast}\Hh(\BC(w_{I_s})) = \sideset{^{\theta}}{_{\SU(r)}^\ast}\Hh( \BC(\sigma_s)) \otimes_{\Ho^\ast_{\SU(r)}(\BSa(\sigma_s))} \Ho_{\SU(r)}^\ast( \BSa(w_{I_s}) ). \]
This example easily generalizes to the case when $I$ is an indexing sequence with the property that given any $s \in \nu(I)$, all occurances of $s$ are in consecutive order. In such a case, one can use sequential multiplication to construct a map $M$ on $\BC(w_I)$ taking values in $\BC(w_{J})$, for some redundancy free sequence $J$ with $\nu(J) = \nu(I)$. 
\end{example}

\section{Appendix: Examples of Twisted Theories} \label{Appendix2}

\medskip
\noindent
In this section, we describe in brief two examples of twisted theories in addition to the one we have already studied. We begin with a simple example of twisted cohomology in the usual sense, and then we construct a twist on equivariant K-theory.  

\medskip
\noindent
Consider equivariant singular cohomology $\Ho_{\SU(r)}[\Z]$ with coefficients in the group-ring of $\Z$, which we identify with Laurent polynomials $\Z[b_0^{\pm 1}]$ on a class $b_0$ in degree zero. In other words
\[ \Ho_{\SU(r)}[\Z] = \Map(\ESU(r), \Ho \wedge \Z_+), \quad \quad  \Ho_{\SU(r)}[\Z]^\ast(pt) = \Ho^\ast_{\SU(r)}(pt)[b_0^{\pm}], \quad |b_0|=0.\]
$\Ho_{\SU(r)}[\Z]$ admits a Galois symmetry $\sigma_{[\Z]}$ that acts by $-1$ on $\Z$. In particular $\sigma_\Z(b_0) = b_0^{-1}$. 

\medskip
\noindent
We define a twist $\theta_r$ on $\Ho_{\SU(r)}[\Z]$ in terms of definition \ref{Twist-6a} by the action
\[ \theta_r  :\OSU(r)_+ \wedge \Ho_{\SU(r)}[\Z] \longrightarrow \Ho_{\SU(r)}[\Z], \]
that is given by the projection map $\OSU(r) \rightarrow \Z$ followed by the canonical action of $\Z$ on $\Ho_{\SU(r)}[\Z]$. It is straightforward to check that $\sigma_\Z \theta \sigma = \theta$. In other words, the twist $\theta$ is invariant under the conjugate Galois symmetry. It is straightforward to check that the twisted cohomology of a space $\BC(w_I)$ is none other than the equivariant singular cohomology of $\BC(w_I)$ with twisted coefficients:
\[ \pi_1(\rho_I) : \pi_1(\BC(w_I)) \longrightarrow \pi_1(\SU(r)) = \Z \subset \Aut_\Z(\Z[b_0^{\pm 1}]). \]
In other words, we have $\sideset{^{\theta_r}}{_{\SU(r)}^\ast}\Ho(\BC(w_I)) = \Ho_{\SU(r)}^\ast(\BC(w_I), \Z[\Z])$. This twisted cohomology of $\BC(w_I)$ is given by imposing the relation $\sum_{l \in \overline{\nu}(I)} \gamma_l = 0$ (see theorem \ref{DomK3-23}) in exchange for shifting the cohomological degree up by one. As such it appears to be an INS-type equivariant theory whose corresponding link homology is a reduced version of triply graded link homology. 

\medskip
\noindent
\begin{remark} \label{gl0}
The above example can be interpreted in terms of the univerersal twisting $\wp(x)$ as defined in claim \ref{Twist-6b}, where we allow a constant term $(1-b_0)$ and set all other variables $b_i$ as $0$. In other words, we can interpret it as ``$sl(0)$-link homology." 
\end{remark}

\bigskip
\noindent
Let us now move on to an example of a twist on $\SU(r)$-equivariant K-theory, $\Ks_{\SU(r)}$. This twist has been of interest lately due to work of Freed-Hopkins-Teleman \cite{AS, FHT} in which they relate this twisted K-theory for arbitrary compact Lie groups G to the Grothendieck group of representations of the smooth Loop group $\LG$. Indeed, one can think of this twisted K-theory of G as an equivariant K-theory for $\LG$, One may develop that idea for a more general class of groups known as Kac-Moody groups, that extend Loop groups. This direction has been explored in \cite{Ki} where these theories go by the name {\em Dominant K-theory}. In this section, we prefer to use the expression twisted K-theory of $\SU(r)$ instead of Dominant K-theory of $\LSU(r)$. 

\medskip
\noindent
To define the twisting of $\Ks_{\SU(r)}$ we start with the construction of a suitable Hilbert space of $\LSU(r)$ representations, and model the representing objects of the twisted theory as a space of Fredholm operators on that Hilbert space. The representations of $\LSU(r)$ we will consider in this section are known as positive-energy representations \cite{PS}. These representations are infinite dimensional Hilbert representations. Since the theory of positive-energy representations requires us to fix a integral level $n >0$, we do so henceforth. 

\medskip
\noindent
Let us get a sense of what these positive-energy representations look like. Consider the (real) Hilbert space $\Tr^r$ defined as the closure of the trignometric functions with values in the real $2r$-dimensional vector space underlying $\C^r$. This space has a (dense) basis:
 \[ \{ e_i \cos(kt), \, \, e_i \sin(st), \quad  k \geq 0, \, \,  s > 0, \, \, 1 \leq i \leq 2r, \, \, 0 \leq t \leq 2\pi \}. \]
 The Euclidean inner product on $\Tr^r$ is given by integrating the standard Euclidean inner product: 
  \[ \langle f(t), g(t) \rangle = \frac{1}{2\pi } \int \langle f(t), g(t) \rangle \, dt. \]
 We may extend the Euclidean inner product on $\Tr^r$ complex linearly to a non-degenerate bilinear form on $\Tr^r \otimes_\R \C$. 
 
 \medskip
 \noindent
 One can now define the $C^*$-algebra $\mathcal{C}$ generated by the Clifford relations:
 \[ f(t) \, g(t) + g(t) \, f(t) = \langle f(t), g(t) \rangle.
 \] 
 Notice that one has a canonical identification of $\Tr^r \otimes_\R \C$ as the completion of Laurent polynomials on $\C^r \otimes_\R \C$:
 \[ \Tr^r \otimes_\R \C = L^2(S^1, \C^r \otimes_\R \C) = \C^r \hat{\otimes}_\R \C [z, z^{-1}], \quad z = e^{i t}.  \]

 \medskip
 \noindent
 Let $J$ denote the complex structure on $\C^r$. Notice that the $\pm i$-eigenspaces of the complex linear extension of $J$ yields an isotropic decomposition: $\C^r \otimes_\R \C = \overline{W} \oplus W$ \footnote{$W$ is canonically isomorphic to $\C^r$ as a $\SU(r)$-representation}. This induces an isotropic decomposition of $\C^r \otimes_\R \C[z,z^{-1}]  = H^r_+ \oplus H^r_-$ with
  \[ H^r_+ = W[z] \oplus z\overline{W}[z], \quad H^r_- = \overline{W}[z^{-1}] \oplus z^{-1}W[z^{-1}]. \]
 
 \noindent
We will denote by $\Lambda^*(H^r_+)$ the irreducible unitary representation of  $\mathcal{C}$ given by the Hilbert completion of the exterior algebra on $H^r_+$ , with $H^r_+ \subset \mathcal{C}$ acting by exterior multiplication, and $H^r_- \subset \mathcal{C}$ acting by extending the contraction operator using the derivation property.  By construction, $\LSU(r)$ preserves the inner product on $\Tr^r$, and hence it acts on $\mathcal{C}$ by algebra automorphisms. Since $\Lambda^*(H^r_+)$ is the unique representation of $\mathcal{C}$\footnote{Strictly speaking, for this to be true, we must first fix a polarization or equivalence class of maximal isotropic subspaces equivalent to $H^r_{\pm}$} , Schur's lemma says that we get a canonical projective action of $\LSU(r)$ on $\Lambda^*(H^r_+)$ that intertwines the action of $\mathcal{C}$ twisted by the action of $\LSU(r)$. The induced map $\rho(r) : \LSU(r) \longrightarrow \PU(\Lambda^*(H^r_+))$ may be lifted over $\SU(\Lambda^*(H^r_+))$ giving rise to the universal level $n=1$ central extension $\tilde{\LSU}(r) \longrightarrow \LSU(r)$. This representation is called the fermionic Fock space representation. 

\medskip
\noindent
\begin{defn} (The universal central extension $\tilde{\LSU}(r)$ and the fermionic Fock space) \label{Fock}

\noindent
Let $\tilde{\LSU}(r)$ denote the universal central extension of $\LSU(r)$ defined by virtue of the homomorphism $\rho(r) : \LSU(r) \longrightarrow \PU(\Lambda^*(H^r_+))$ above. This central extension $\tilde{\LSU}(r)$ splits over the constant loops $\SU(r) \subset \LSU(r)$. Since the representation ring of $\SU(r)$ is generated by the representations $\Lambda^k(\C^r)$, it follows from the construction that $\Lambda^*(H_+)$ contains all irreducible representations of $\SU(r)$ under the restriction $\SU(r) \subset \tilde{\LSU}(r)$. It is also straightforward to see that $\Lambda^*(H_+^r) = \Lambda^*(H_+^s) \hat{\otimes} \Lambda^*(H^t_+)$ under the central extension of the diagonal map $\LSU(s) \tilde{\times} \LSU(t) \longrightarrow \tilde{\LSU}(r)$, where $r = s+t$. 
\end{defn}

\medskip
\begin{remark} \label{Galois6}
Let $\T$ denote the rotation group acting on $\LSU(r)$ by reparametrizing $S^1$. This action lifts to an action on $\tilde{\LSU}(r)$. Furthermore, the action of $\tilde{\LSU}(r)$ on $\Lambda^*(H^r_+)$ extends to an action of the group $\T \ltimes \tilde{\LSU}(r)$. Let $x_1,x_2, \ldots, x_r$ denote the diagonal characters of the standard representation of $\SU(r)$ on $\C^r$, and let $u$ denote the central character of $\tilde{\LSU}(r)$. Also, let $q$ denote the fundamental character of $\T$. Then it is easy to see that the character of the fermionic Fock space is given by:
\[ \mbox{Ch}(\Lambda^*(H^r_+)) = u \, \prod_{m=0}^\infty \prod_{i=1}^r (1+x_i q^m)(1+x^{-1}_i q^{m+1}). \]
One may define a complex linear Galois symmetry $\sigma_{\Ks}$ on the fermionic Fock space that has the property $\sigma_{\Ks}(x_i) = x_i^{-1}q$ and $\sigma_{\Ks}(q) = q$. Then $\sigma_{\Ks}$ induces a symmetry on the space of Fredholm operators on the fermionic Fock space (by conjugation with $\sigma_{\Ks})$. 
\end{remark}
 
 \medskip
\begin{defn} (Level $n$ positive-energy representations)

\noindent
We define an irreducible level $n$ positive-energy representation of $\tilde{\LSU}(r)$ to be any irreducible $\tilde{\LSU}(r)$-representation that is a sub-representation of the $n$-fold (completed) tensor product of $\Lambda^*(H^r_+)$ (see \cite{PS} Thm. 10.6.4). There are only finitely many irreducible positive-energy representations of level $n$, and that any extensions of two irreducible positive-energy representations splits \cite{PS}. In particular, one has a semi-simple category of positive-energy level $n$ representations of $\tilde{\LSU}(r)$. 
\end{defn}

\medskip
\noindent
Before we define twisted K-theory, let us first recall that usual two-periodic K-theory is represented by homotopy classes of maps into the space of Fredholm operators $\gF(\gH)$ (with the norm topoogy) in a separable Hilbert space $\gH$. This space of operators has the homotopy type of the infinite Grassmannian $\Z \times \BSU$. The theorem of Bott periodicity ensures that $\Z \times \BSU$ is an infinite loop space, thereby defining a (two-periodic) cohomology. 

\medskip
\noindent
The construction K-theory described above has an equivariant analog as well \cite{Se}. Given a compact Lie group G, the object that represents G-equivariant K-theory is the G-space of Fredholm operators (under $G$-conjugation) in a Hilbert G-representation that contains all G-representations, with infinite multiplicity. With this as motivation, we define

\medskip
\begin{defn} (The $\LSU(r)$-equivariant theory $\sideset{^n}{_{\SU(r)}}\Ks$) \label{Dom}

\noindent
Let $\gH_n$ denote the Hilbert space completion of countable copies of all level $n$ positive-energy representations of $\tilde{\LSU}(n)$. Let $\gF(\gH_n)$ denote the space of Fredholm operators on $\gH_n$. By choosing a suitable variation of the norm-topology (see \cite{AS} section 3), the underlying homotopy type of $\gF(\gH_n)$ is given by the infinite loop-space $\Z \times \BSU$ and the group $\tilde{\LSU}(r)$ admits a continuous action on $\gF(\gH_n)$ by conjugation of operators which factors through $\LSU(r)$. Furthermore, the infinite loop-space structure on $\gF(\gH_n)$ is compatible with respect to this action (see \cite{AS} sections 3, 4) giving rise to an $\LSU(r)$-equivariant cohomology theory denoted by $\sideset{^n}{_{\SU(r)}}\Ks$. 
\end{defn}

\medskip
\begin{remark} \label{Fock2}
By definition \ref{Fock}, $\gH_n$ contains all $\SU(r)$-representations with infinite multiplicity, hence the $\SU(r)$-equivariant theory underlying $\sideset{^n}{_{\SU(r)}}\Ks$ is simply $\SU(r)$-equivariant K-theory $\Ks_{\SU(r)}$. In particular, $\sideset{^n}{_{\SU(r)}}\Ks$ represents a twist $\theta_r$ on $\Ks_{\SU(r)}$ indexed by the level $n$ representations of $\LSU(r)$. The twists of $\Ks_{\SU(r)}$ have been classified by $\Ho^3_{\SU(r)}(\SU(r))$ \cite{AS}, where $\SU(r)$ is seen as a $\SU(r)$-space under conjugation. As such, $\theta = \{ \theta_r \}$ represents the class $n B_1$ (see theorem \ref{Univdiff}). 

\medskip
\noindent
Notice that one also has a Galois symmetry $\sigma$ on $\LSU(r)$ by pointwise complex conjugation. This action is compatible with the symmetry $\sigma_{\Ks}$ on $\sideset{^n}{_{\SU(r)}}\Ks$ induced by the Galois symmetry on the fermionic Fock space (see remark \ref{Galois6}). In other words, the twist $\theta$ is invariant under the conjugate symmetry, giving rise to a Galois symmetry on the twisted equivariant K-theory groups of $\BC(w_I)$. 
\end{remark}

\medskip
\begin{example} 
Consider the simplest space of Broken symmetries $\SU(r) \times_{T^r} T^r$. As with twisted equivariant cohomology, one may easily compute its twisted K-theory 
\[ \sideset{^n}{_{\SU(r)}^{r}}\Ks(\SU(r) \times_{T^r} T^r) = \frac{\Z[x_1^{\pm}, \ldots, x_r^{\pm}]}{\langle x_1^n-1, \ldots, x_r^n-1 \rangle}, \quad \sideset{^n}{_{\SU(r)}^{r-1}}\Ks(\SU(r) \times_{T^r} T^r) = 0. \]
The action of the conjugate Galois symmetry $\sigma$ is given by $\sigma(x_i) = x_i^{-1}$. 
\end{example}

\medskip
\noindent
As in the case of the twisted theory $\{ \Hh, \theta_\wp \}$, computing $\sideset{^n}{_{\SU(r)}}\Ks$ on more general spaces $\BC(w_I)$ and verifying that one has an INS-type theory is a non-trivial problem. This remains work in progress.

\pagestyle{empty}
\bibliographystyle{amsplain}

\begin{thebibliography}{10}
\bibitem{A} 
J. F. \ Adams, \emph{Stable Homotopy and Generalised Homology}, Univ. of Chicago Press, 1974.
\bibitem{AS} M. Atiyah, G. Segal, \textit{Twisted K-theory}, available at {\tt math.KT/0407054}.
\bibitem{CK} N. Castellana, N. Kitchloo, \textit{A homotopy decomposition of the adjoint representation for Lie groups}, Math. Proc. Camb. Phil. Soc., 133, (2002), 399--409.
\bibitem{FHT} D. S. Freed, M. J. Hopkins, C. Teleman, \textit{Loop groups and twisted K-theory, III}, Annals of Math., Vol. 174 (2), 947--1007, 2011. 
\bibitem{Go} T. Mej\'{i}a Gomez, \textit{The Rasmussen spectral sequence and Symmetry Breaking}, preprint, 2022. 
\bibitem{GHMN} 
E. Gorsky, M. Hogancamp, A. Mellit, K. Nakagane, \emph{Serre duality for Khovanov-Rozansky homology}, arXiv:1902.08281, 2019. 
\bibitem{HY} J. Hahn, A. Yuan, \textit{Multiplicative structure in the stable splitting of $\Omega SL_n(\C)$}, Adv. Math. 348, 412--455, 2019. 
\bibitem{I}
M. Inoue, \emph{Odd primary Steenrod algebra, additive formal group laws and modular invariants}, J. Math. Soc. Japan, Vol. 58, No. 2, 311--333, 2006. 
\bibitem{K} 
M. Khovanov, \emph{Triply-graded link homology and Hochschild homology of Soergel bimodules}, International Journal of Math., vol. 18, no.8, 869--885, 2007, arXiv:math.GT/0510265.
\bibitem{KR}
M. Khovanov, L. Rozansky, \emph{Positive half of the Witt algebra acts on triply graded Link homology}, arXiv:1305.1642. 
\bibitem{K1} M. Khovanov, \emph{Link homology and categorification}, Proceedings of the ICM, 2006, Vol II., 989--999, 2006. 
\bibitem{KR2}
M. \ Khovanov, L, Rozansky, \emph{Matrix factorizations and link homology}, Fundamenta Mathematicae, vol.199, 1--91, 2008. 
\bibitem{Ki} N. Kitchloo, \textit{Dominant K-theory and Integrable highest weight representation of Kac-Moody groups}, Advances in Math., 221, (2009), 1227-1246. 
\bibitem{Ki1} N. Kitchloo, \textit{Symmetry breaking and Link homologies I}, preprint, 2019. 
\bibitem{Ka} D. Krasner, \emph{Integral HOMFLY-PT and $sl(n)$-Link Homology}, International Journal of Math., vol. 2010. 
\bibitem{Ku} S. Kumar, \emph{Kac-Moody Groups, their Flag Varieties, and Representation Theory}, Vol. 204, Progress in Math., Birkh\"{a}user, 2002. 
\bibitem{LMS}
L. \ G. Lewis, P. May, M. Steinberger \emph{Equivariant Stable homotopy theory}, Springer Lecture Notes in Math., Vol. 1213, 1986. 
\bibitem{KM} N. Kitchloo, J. Morava, \textit{Thom prospectra and Loop group representations}, Elliptic Cohomology, London Math. Soc. Lecture Note Series (342), Cambridge Univ. Press, (2007), 214-238. 
L. \ G. \ Lewis, P. \ May, M. \ Steinberger \emph{Equivariant Stable homotopy theory}, Springer Lecture Notes in Math., Vol. 1213, 1986. 
\bibitem{M} J. Morava, \emph{On the complex cobordism ring as a Fock representation}, Homotopy theory and related topics (Kinosaki, 1988), 184--204, Lecture Notes in Math., 1418, Springer, Berlin, 1990.
\bibitem{Le}
N. Lebedinsky, \emph{Gentle introduction to Soergel bimodules I}, arXiv:1702.00039v1, 2017.
\bibitem{MT}
M. Mimura, H. Toda \emph{Topology of Lie Groups, I and II}, Trans. of Math. Monographs, Vol, 91, AMS, 2000. 
\bibitem{PS} A. Presley, G. Segal, \textit{Loop Groups}, Oxford University Press, 1986. 
\bibitem{J2}
J. Rasmussen, \emph{Some differentials on Khovanov-Rozansky homology}, Geom. Top. Vol. 19, no 6, 3031--3104, 2015. 
\bibitem{R}
R. Rouquier, \emph{Categorification of $sl_2$ and braid groups}, Contemporary Math. Vol 406, 137--167, 2006. 
\bibitem{R2}
R. Rouquier, \emph{Khovanov-Rozansky homology and 2-braid groups}, Contemp. Math. Vol 684, 147--157, 2017.
\bibitem{Se} G. Segal, \textit{Equivariant K-theory}, Publ. Math. I.H.E.S., Paris, 34, (1968).
\bibitem{S}
L. Smith, \emph{Lectures on the Eilenberg--Moore spectral sequence}, Lecture Notes in Mathematics 134, Berlin, New York, 1970. 
\bibitem{Sc}
S. \ Schwede, \emph{Global Homotopy Theory}, New Math. Monographs, 34, Cambridge Univ. Press, 2018.


\bibitem{WW} B. Webster, G. Williamson, \emph{A geometric model for Hochschild homology of Soergel bimodules}, (English
summary) Geom. Topol. 12, no. 2, 1243--1263, 2008. See also: arXiv:0707.2003v2. 
\bibitem{Wi} C. Wilkerson, \emph{Self-maps of classifying spaces}, in Lecture notes in Math., Vol. 418, 150--157, 1974.
\bibitem{W}
R. M. W. Wood, \emph{Differential Operators and the Steenrod Alegebra}, Proc. London Math. Soc. 194--220, 1997. 

\end{thebibliography}
\providecommand{\bysame}{\leavevmode\hbox
to3em{\hrulefill}\thinspace}

\end{document}